\documentclass[12pt]{article}

\usepackage{amsmath,amssymb,amsthm, amsfonts}
\usepackage{fancyhdr}
\usepackage{listings}
\usepackage{hyperref}
\hypersetup{
  colorlinks   = true, 
  urlcolor     = blue, 
  linkcolor    = blue, 
  citecolor   = blue 
}
\usepackage{multicol}
\usepackage{graphicx}
\usepackage{url}
\usepackage[mathlines]{lineno}
\usepackage{tikz}
\usepackage{blkarray}
\usepackage{color}
\usepackage{subfig}
\definecolor{red}{rgb}{1,0,0}

\definecolor{blue}{rgb}{0,0,1}

\definecolor{green}{rgb}{0,.6,0}

\usepackage{verbatim}
\usepackage{tkz-graph}
\usepackage{graphicx}
\graphicspath{ {images/} }
\usetikzlibrary{positioning,graphs,graphs.standard,matrix}
\usepackage{tkz-fct}
\usepackage{float}

\usepackage{pgfplots}
\pgfplotsset{width=7cm,compat=1.13}

\newtheorem{thm}{Theorem}[section]
\newtheorem{cor}[thm]{Corollary}

\newtheorem{prop}[thm]{Proposition}
\newtheorem{conj}[thm]{Conjecture}
\newtheorem{obs}[thm]{Observation}

\theoremstyle{definition}
\newtheorem{rem}[thm]{Remark}

\theoremstyle{definition}
\newtheorem{defn}[thm]{Definition}

\theoremstyle{definition}
\newtheorem{ex}[thm]{Example}

\numberwithin{figure}{section}

\newcommand{\bit}{\begin{itemize}}
\newcommand{\eit}{\end{itemize}}
\newcommand{\ben}{\begin{enumerate}}
\newcommand{\een}{\end{enumerate}}
\newcommand{\beq}{\begin{equation}}
\newcommand{\eeq}{\end{equation}}
\newcommand{\bea}{\begin{eqnarray*}} 
\newcommand{\eea}{\end{eqnarray*}}
\newcommand{\bpf}{\begin{proof}}
\newcommand{\epf}{\end{proof}\ms}
\newcommand{\bmt}{\begin{bmatrix}}
\newcommand{\emt}{\end{bmatrix}}
\newcommand{\ms}{\medskip}

\title{Techniques for determining equality of the maximum nullity and the zero forcing number of a graph }
\author{Derek Young}

\begin{document}
 
\maketitle
\begin{abstract} 
It is known that the zero forcing number of a graph is an upper bound for the
maximum nullity of the graph (see \cite{MR2388646}).  In this paper, we search
for characteristics of a graph that guarantee the maximum nullity of the graph
and the zero forcing number of the graph are the same by studying a variety of
graph parameters which bound the maximum nullity of a graph below. In
particular, we introduce a new graph parameter which acts as a lower bound for
the maximum nullity of the graph. As a result, we show that the Aztec Diamond
graph's maximum nullity and zero forcing number are the same.  Other graph
parameters that are considered are a Colin de Verdi\'ere type parameter and the
vertex connectivity. We also use matrices, such as a divisor matrix of a graph
and an equitable partition of the adjacency matrix of a graph, to establish a
lower bound for the nullity of the graph's adjacency matrix. 
\bigskip
\end{abstract}



\section{Introduction}
\label{sec:introduction}
The maximum nullity over a set of matrices that can be
described by a graph has been well studied (see \cite{MR2388646, MR3010007,MR2350678, MR3013937}).
While determining the maximum nullity over a set of matrices described
by a graph is not easy to compute, there are graph parameters that allow us to
bound the maximum nullity. For some graphs, these bounds are enough to determine
the maximum nullity. Unfortunately, the bounds available are not enough
to determine the maximum nullity for all graphs. 
\bigskip

The zero forcing number, described in detail below, is an upper bound for
maximum nullity. The problem of characterizing graphs that have the property
that the maximum nullity of the graph is equal to zero forcing number of the
graph was first posed in \cite{MR2388646}. While this problem is still open,
there are many families of graphs that have their maximum nullity equal to
their zero forcing number. A list of families of graphs having this property
can be found in \cite{MRcatalog} including trees, cycles, complete graphs,
complete bipartite graphs, completely subdivided graphs, and graphs with less
than $ 8 $ vertices.  The zero forcing number of a graph can be computed by
using mathematical software.  However, determining the maximum nullity of a
graph is a challenging problem. 
\bigskip

A \textit{graph}, denoted by $ G $, consists of a set $ V(G) $ called a
\textit{vertex set} and an \textit{edge set} $ E(G)  $ where the edge set
contains two element subsets of the vertex set.  For convenience, when $
\{v,u\} \in E(G) $ we may drop the brackets and write $ vu $. The
\textit{order} of a graph, denoted by $ |G| $, is the number of vertices in the
graph.   The spectrum of a symmetric matrix $ A $, denoted by $
\operatorname{spec}(A) $, is the multiset of eigenvalues of $ A $. The nullity
of a symmetric matrix, denoted by $ \operatorname{null}(A) $, is the number of
times zero occurs in $ \operatorname{spec}(A) $.  The rank of a symmetric
matrix $ A $, denoted by $ \operatorname{rank}(A) $, is the dimension of the
vector space spanned by the rows of $ A $. 
\bigskip

The \textit{set of symmetric matrices of a graph $ G $ over a field $
\mathcal{F} $}, denoted by $ \mathcal{S}(\mathcal{F}, G) $, is the set of
symmetric matrices $ A = [a_{ij}] $ having the same off-diagonal nonzero
pattern as the adjacency matrix of $ G $ (for $i \neq j,  a_{ij} \neq 0 \iff ij
\in E(G) $) with free diagonal entries ($ a_{ii} \in \mathcal{F} $). The
adjacency matrix of a graph $ G $, denoted by $ A(G) $, is the matrix in
$ \mathcal{S}(\mathbb{R},G) $ with $ a_{ii} = 0 $ and $ a_{ij} = 1$ for $ i
\neq j $. The \textit{maximum nullity of a graph $ G $ over a field $
\mathcal{F} $}, denoted by $ \operatorname{M}(\mathcal{F},G) $ is the maximum
nullity over $ \mathcal{S}(\mathcal{F},G) $.  Let $ A \in
\mathcal{S}(\mathbb{R}, G) $.  Because the diagonal of a matrix $ A \in
\mathcal{S}(\mathbb{R},G) $ is unrestricted, all eigenvalues of $ A $ are real, the
algebraic and geometric multiplicity of $ A $ are equal, and the nullity of $ A -
\lambda I $ is the multiplicity of $ \lambda $ as an eigenvalue of $ A $.
Thus, the maximum multiplicity over matrices in $ \mathcal{S}(\mathbb{R},G) $ is
the same as the maximum nullity over $ \mathcal{S}(\mathbb{R},G) $.  The
\textit{minimum rank of a graph $ G $ over a field $ \mathcal{F} $}, denoted by
$ \operatorname{mr}(\mathcal{F},G) $, is the minimum rank over $
\mathcal{S}(\mathcal{F},G) $. Whenever the field is not specified, the field is
understood to be the real numbers $ \mathbb{R} $. Observe that $
\operatorname{mr}(\mathcal{F}, G) + \operatorname{M}(\mathcal{F},G) = |G|$,
where $ |G| $ is defined to be the number of vertices in the graph $ G $. This
makes solving for $ \operatorname{M}(\mathcal{F},G) $ equivalent to solving the
associated minimum rank problem.  See \cite{MR2350678} for a discussion on the
motivation of the minimum rank problem. 
\bigskip

Let $ Z $ be a subset of $ V(G) $ such that every vertex in $ Z $ is colored
blue and all other vertices are colored white.  The \textit{color
change rule} for zero forcing is: A blue vertex can change a
white vertex blue if the white vertex is the only white vertex adjacent to the
blue vertex. (Vertices $ v $
and $ u $ are said to be \textit{adjacent} if and only if $ \{v,u\} \in E(G)
$.) In this case, we say that the blue vertex forced the white vertex
blue.  A \textit{zero forcing set} is a subset of $ V(G) $ such that after
applying the color change rule until no more changes are possible, all vertices
in $ G $ are colored blue. The \textit{zero forcing number} of a graph $ G $,
denoted by $ \operatorname{Z}(G) $, is the minimum cardinality over all zero
forcing sets. A \textit{chronological list of forces} is a sequence of forces
performed in the given order. The term zero forcing refers to forcing entries
in the null vector to be zero, which leads to the relationship that the maximum nullity of a graph is bounded above by the zero forcing number of the graph.  

\begin{prop}{\rm\cite[Proposition 2.4]{MR2388646}}
\label{prop:M_is_less_than_Z_over_any_field}
Let $ G $ be a graph and let $ \mathcal{F} $ be a field. Then 
\[
\operatorname{M}(\mathcal{F},G) \leq \operatorname{Z}(G).
\]
\end{prop}

Note that $A(G) - \lambda I$ with $\lambda \in \mathbb{Z}$ can be viewed as a
matrix over any field $\mathcal{F}$. When $ \mathcal{F} $ is of characteristic
$p$, each integer interpreted as its residue modulo $p$. Thus $A(G) - \lambda I
\in \mathcal{S}(\mathcal{F},G)$. When we view $ A \in \mathcal{F}^{n x n} $ we
write $ \operatorname{rank}(\mathcal{F},A) $ for the rank which may depend on $
\mathcal{F} $. An \textit{optimal matrix over a field} $ \mathcal{F} $ is a
matrix $ A \in \mathcal{S}(G) $ such that $\operatorname{rank}(\mathcal{F},A) =
\operatorname{mr}(\mathcal{F},G) $. We say that an integer matrix $ A \in
\mathcal{S}(\mathcal{F},G) $ that has entries $ -1,0,1 $ on the off diagonal is
\textit{universally optimal}  if for all fields $ \mathcal{F} $,
$\operatorname{rank}(\mathcal{F},A) = \operatorname{mr}(\mathcal{F},G) $. The
minimum rank of a graph $ G $ is said to be \textit{field independent} if for
all fields $\mathcal{F}$, $ \operatorname{mr}(\mathcal{F},G) =
\operatorname{mr}(G) $. The minimum rank problem over fields other than the
real numbers was studied as early as 2004 by Wayne Barrett, Hein van der Holst,
and Raphael Loewy in \cite{MR2111528}. In 2009, DeAlba, et. al \cite{MR2530143}
used universally optimal matrices to establish minimum rank field independence
for many graphs listed in \cite{MRcatalog}.  

\begin{prop}{\rm\cite[Corollary 2.3]{MR2530143}}
\label{prop:field_rank_less_than_real_rank}
If $ A \in \mathbb{Z}^{n \times n} $, then $
\operatorname{rank}(\mathbb{Z}_p , A) \leq \operatorname{rank}(A) $ for every prime $ p $.
\end{prop}

\begin{cor}
\label{cor:M_equal_Z_field_independence}
Let $ G $ be a graph having the property that for some $ \lambda \in
\mathbb{Z} $, $\operatorname{rank}(A(G)-\lambda I) =|G| -  \operatorname{Z}(G)
$, or equivalently, $ \operatorname{null}(A(G) - \lambda I) =
\operatorname{Z}(G). $
Then the minimum rank of $ G $ is field independent and $ A(G) - \lambda I $
is universally optimal, and $ \operatorname{M}(\mathcal{F},G) =
\operatorname{Z}(G) $ for all fields $ \mathcal{F} $.
\end{cor}

\begin{proof}
\label{pf:M_equal_Z_field_independence}
By \hyperref[prop:M_is_less_than_Z_over_any_field]{Proposition
\ref{prop:M_is_less_than_Z_over_any_field}}, $ |G| -
\operatorname{mr}(\mathcal{F},G) = \operatorname{M}(\mathcal{F},G) \leq
\operatorname{Z}(G) $ and by
\hyperref[prop:field_rank_less_than_real_rank]{Proposition \ref{prop:field_rank_less_than_real_rank}} we have 
$ \operatorname{rank}(\mathcal{F},A(G) - \lambda I ) \leq
\operatorname{rank}(A(G) - \lambda I ) $, so $ \operatorname{null}(A(G) -\lambda I) \leq
\operatorname{null}(\mathcal{F}, A(G) - \lambda I)$. It follows that 

\[
\operatorname{Z}(G) \geq \operatorname{M}(\mathcal{F}, G) \geq
\operatorname{null}(\mathcal{F}, A(G) - \lambda I) \geq
\operatorname{null}(A(G) - \lambda I) = \operatorname{Z}(G).
\]

Therefore, $ \operatorname{mr}(\mathcal{F},G) = 
\operatorname{rank}(\mathcal{F},A(G) + \lambda I) = |G|- \operatorname{Z}(G) $
which shows that $ G $ has field independent minimum rank and $ A(G) + \lambda
I $ is universally optimal.
\end{proof}

\begin{obs}
\label{obs:not_field_independent}
Let $ G $ be a graph.  If there exists a prime $ p $ such that $
\operatorname{mr}(\mathbb{Z}_p,G) \neq \operatorname{mr}(G)$ then $ G $ does
not have field independent minimum rank.
\end{obs}

A \textit{generalized Petersen Graph}, denoted by $ P(n,k) $, is a graph having
a labeled vertex set $ \{ u_0, u_1, \dots u_{n-1}, v_0, v_1, \dots, v_{n-1} \}
$ and edge set $ \big\{\{u_iu_{i+1 \bmod n }\}, \{v_iv_{i+k \bmod n }\},
\{u_iv_i\} : i = 0,1,2,\\\dots,n-1 \big\}$, for $ n \geq 3 $ and $ k $ a positive
integer less than $ \lfloor \frac{n}{2} \rfloor $. In \cite{MR3760976}, the
adjacency matrix was used to show that the maximum nullity is equal to the zero
forcing number for certain generalized Petersen graphs.

\begin{thm}{\rm\cite[Theorem 2.4]{MR3760976}}
\label{thm:GP_M_equal_Z}
Let $ r $ be a positive integer. Then 
\[
\operatorname{M}(P(15r,2))= \operatorname{Z}(P(15r,2)) = 6 \quad \text{and}
\quad \operatorname{M}(P(24r,5))= \operatorname{Z}(P(24r,5)) = 12
\]
and the maximum nullity is attained by the adjacency matrix.
\end{thm}

\begin{cor}
\label{cor:gp_field_independent}
Let $ r $ be a positive integer. Then the two subfamilies $ P(15r,2) $ and $
P(24r,5) $ have field independent minimum rank with universally optimal
matrices. Moreover, for all fields $ \mathcal{F} $,
\[
\operatorname{M}(\mathcal{F}, P(15r,2))= \operatorname{Z}(P(15r,2)) \quad \text{and}
\quad \operatorname{M}(\mathcal{F}, P(24r,5))= \operatorname{Z}(P(24r,5)).
\]
\end{cor}

The \textit{Cartesian product} of the graphs $ G $ and $ H $, denoted by $ G
\square H $, has vertex set $ \{ (v,w) | v \in V(G), w \in
V(H) \} $ and edge set 
\[
\{ (v_1,w_1)(v_2,w_2) \, | \, ( v_1 = v_2
\text{ and } w_1w_2 \in E(H) ) \text{ or } ( v_1v_2 \in E(G) \text{ and } w_1
=w_2 ) \}.
\]

\begin{thm}{\rm\cite[Theorem 3.8]{MR2388646}}
\label{thm:M_equal_cartesian_product}
Let $ k \geq 3 $. Then $ \operatorname{M}(C_k \square P_t) =
\operatorname{Z}(C_k \square P_t) = \min\{k,2t\} $.
\end{thm}
 
\begin{ex}
\label{ex:generalized_petersen_not_field_independent}
By 
\hyperref[thm:M_equal_cartesian_product]{Theorem
\ref{thm:M_equal_cartesian_product}}, $ \operatorname{M}(
C_7 \square P_2) = 4 $ which implies $ \operatorname{mr}(
C_7 \square P_2) = 10 $. By computation via SageMath (see \cite{ecgsage}), there does not
exist a matrix in $ \mathcal{S}(\mathbb{Z}_2,C_7 \square P_2) $
having rank equal to 10.  Therefore by
\hyperref[obs:not_field_independent]{Observation
\ref{obs:not_field_independent}}, $ C_7 \square P_2 $ does not have field
independent minimum rank. 
\end{ex}

\hyperref[ex:generalized_petersen_not_field_independent]{Example
\ref{ex:generalized_petersen_not_field_independent}} shows that the generalized
Petersen graphs do not have field independent minimum rank field independent
since $ C_7 \square P_2 $ is isomorphic to $ P(7,1) $. It is known that $ C_n \square
P_t $ does not have field independent minimum rank (see {\rm\cite[Example
3.5]{MR2530143}}).

\section{An application of the nullity of a graph}
\label{sec:the_Nullity_of_a_Graph}
In this section, we introduce a new graph parameter based on the nullity of the
adjacency matrix that acts as a lower bound for the maximum nullity of the
graph.  We then apply this to the Aztec diamond graphs and to some circulants
to compute the maximum nullity and the zero forcing number, and show that they
have field independent minimum rank with the adjacency matrix as a universally
optimal matrix.
\bigskip

A \textit{general graph}  is a graph that may contain loops (edges of the form
$ vv $) and/or multi-edges (two edges containing the same vertices $ u $ and $
v $ are called \textit{multi-edges}). Let $ G $ be a general graph and let $
v,u \in V(G) $. The neighborhood of $ v $ in a general graph $ G $, denoted by
$ N_G(v) $, is a multiset containing vertices of $ V(G) $ such that $ k $
copies of $ u $ are in $ N_G(v) $ if and only if there are $ k $ copies of $ uv
$ in $ E(G) $. Let $ X $ and $ Y $ be multisets containing elements of $ V(G)
$. The general graph $ G_{v+X} $ is obtained from $ G $ by adding one edge $ vw
$ for each $ w \in N_G(x) $ and for every $ x \in X $ (see
\hyperref[fig:added_substracted_vertices]{Figure
\ref{fig:added_substracted_vertices}}). Let $ v \in X $ and $ y \in Y $.
Suppose $ N_{G_{v+X}}(y) \subseteq N_{G_{v+X}}(v) $. Then the general graph $
G_{v+X-Y} $ is obtained from $ G_{v+X} $ by deleting one edge $ vw $ for each $
w \in N_{G_{v+X}}(y) $ and for every $ y \in Y $ (see
\hyperref[fig:added_substracted_vertices]{Figure
\ref{fig:added_substracted_vertices}}). In the case that $ X $ and $ Y $
consists of a single vertex $ x $ or $ y $, we write $ G_{v+x} $ or $ G_{v+x-y}
$.
\bigskip

We define a color change rule as follows: 
In a graph $ G $, having each vertex colored red or white, a white vertex $ u
$ can be colored red if there exists a white vertex $ v$ and multisets of
white vertices $X , Y $ such that 
\begin{enumerate}
\item $ u \notin \{v\} \cup X \cup Y $, and
\item $ N_{G_{u+U_k}}(u) =
N_{G_{v+X-Y}}(v) $ 
\end{enumerate}
for some nonnegative integer $ k $ and the multiset $ U_k $ containing $ k $
copies of $ u $, (whenever $ k = 0 $, $ U_k $ is the empty set and $
N_{G_{u+U_k}}(u) = N_{G}(u) $). In this case we say that $ u $ can be colored
red by $
(v,X,Y,k) $.
\bigskip

\begin{ex}
\label{ex:red_set}
\hyperref[fig:added_substracted_vertices]{Figure
\ref{fig:added_substracted_vertices}} illustrates the process of creating $ G_{1+4-0} $. Moreover, vertices $ 1 $ and $ 3 $ have the
same neighborhood in $ G_{1+4-0} $, so vertex $ 3 $ can be colored red
in $ G$ by $ (1,\{4\},\{0\},0) $. We can also color vertex $ 5 $ red. Consider
the general graph $ G_{1+4-2} $ in which vertices $ 1 $ and $ 5 $
have the same neighborhood in $ G_{1+4-2}$. 
\end{ex}

\begin{figure}[t]
\begin{center}
\begin{tikzpicture}[scale=0.83,font=\scriptsize]
\GraphInit[vstyle=Normal]
%
\tikzset{VertexStyle/.style = {shape = circle, font=\large}}
\Vertex[L=\hbox{$G$},x=0cm,y=3cm]{G}
\tikzset{VertexStyle/.style = {shape = circle, draw=black,
minimum size=14pt, inner sep=0.5pt}}
\Vertex[L=\hbox{${0}$},x=2cm,y=3cm]{v0}
\Vertex[L=\hbox{${1}$},x=-0.2cm,y=0.3cm]{v1}
\Vertex[L=\hbox{${2}$},x=4.2cm,y=0.3cm]{v2}
\Vertex[L=\hbox{${3}$},x=2cm,y=2cm]{v3}
\Vertex[L=\hbox{${4}$},x=1cm,y=1cm]{v4}
\Vertex[L=\hbox{${5}$},x=3cm,y=1cm]{v5}
\Edge[](v0)(v1)
\Edge[](v0)(v3)
\Edge[](v0)(v2)
\Edge[](v1)(v2)
\Edge[](v1)(v4)
\Edge[](v2)(v5)
\Edge[](v3)(v4)
\Edge[](v3)(v5)
\Edge[](v4)(v5)
\end{tikzpicture}
\hspace{1cm}
\begin{tikzpicture}[scale=0.83,font=\scriptsize]
\GraphInit[vstyle=Normal]
%
\tikzset{VertexStyle/.style = {shape = circle, font=\large}}
\Vertex[L=\hbox{$G_{1+4}$},x=0cm,y=3cm]{G}
\tikzset{VertexStyle/.style = {shape = circle, font=\large}}
\tikzset{VertexStyle/.style = {shape = circle, draw=black,
minimum size=14pt, inner sep=0.5pt}}
\Vertex[L=\hbox{${0}$},x=2cm,y=3cm]{v0}
\Vertex[L=\hbox{${1}$},x=-0.2cm,y=0.3cm]{v1}
\Vertex[L=\hbox{${2}$},x=4.2cm,y=0.3cm]{v2}
\Vertex[L=\hbox{${3}$},x=2cm,y=2cm]{v3}
\Vertex[L=\hbox{${4}$},x=1cm,y=1cm]{v4}
\Vertex[L=\hbox{${5}$},x=3cm,y=1cm]{v5}
\Edge[](v0)(v1)
\Edge[](v0)(v3)
\Edge[](v0)(v2)
\Edge[](v1)(v2)
\Edge[](v1)(v4)
\Edge[](v2)(v5)
\Edge[](v3)(v4)
\Edge[](v3)(v5)
\Edge[](v4)(v5)
\Edge[](v1)(v5)

\Edge[style=bend left](v1)(v3)
\Loop[dist=1cm,dir=WE,style={thick,-}](v1)
\end{tikzpicture}
\hspace{1cm}
\begin{tikzpicture}[scale=0.83,font=\scriptsize]
\GraphInit[vstyle=Normal]
%
\tikzset{VertexStyle/.style = {shape = circle, font=\large}}
\Vertex[L=\hbox{${G_{1+4-0}}$},x=0cm,y=3cm]{G}
\tikzset{VertexStyle/.style = {shape = circle, font=\large}}
\tikzset{VertexStyle/.style = {shape = circle, draw=black,
minimum size=14pt, inner sep=0.5pt}}
\Vertex[L=\hbox{${0}$},x=2cm,y=3cm]{v0}
\Vertex[L=\hbox{${1}$},x=-0.2cm,y=0.3cm]{v1}
\Vertex[L=\hbox{${2}$},x=4.2cm,y=0.3cm]{v2}
\Vertex[L=\hbox{${3}$},x=2cm,y=2cm]{v3}
\Vertex[L=\hbox{${4}$},x=1cm,y=1cm]{v4}
\Vertex[L=\hbox{${5}$},x=3cm,y=1cm]{v5}
\Edge[](v0)(v1)
\Edge[](v0)(v3)
\Edge[](v0)(v2)
\Edge[](v1)(v4)
\Edge[](v2)(v5)
\Edge[](v3)(v4)
\Edge[](v3)(v5)
\Edge[](v4)(v5)
\Edge[](v1)(v5)

\end{tikzpicture}
\end{center}
\caption{This shows the graph $ G_{1+4-0}. $}
\label{fig:added_substracted_vertices}
\end{figure}
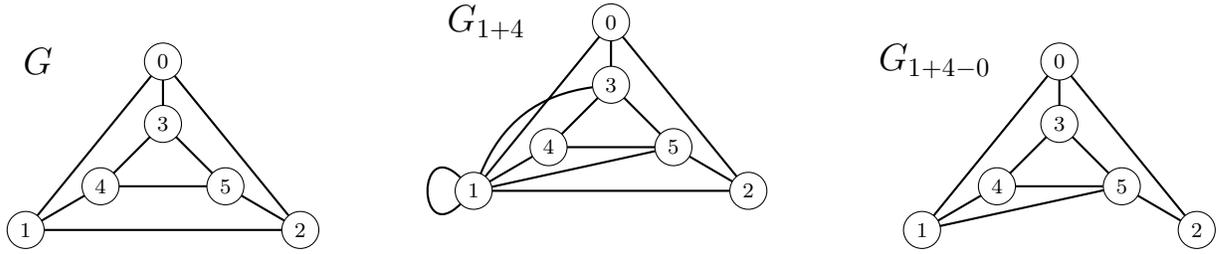

A set of red vertices is called a \textit{red set}, denoted by $ \mathcal{R}
$, if the vertices $ v_1,v_2, \dots, v_t $ of $ \mathcal{R} $ can be
sequentially colored red. The \textit{nullity of a graph} $ G $, denoted by $
\operatorname{null}(G) $, is the maximum cardinality over the set of all red
sets.
\bigskip

\begin{obs}
\label{obs:ccr_equation}
Let $ u,v $ be white vertices of $ V(G) $, $ X $ and $ Y $ be multisets
containing white vertices of $ V(G) $, and $ k $ be a nonnegative
integer. Then $ u $ can be colored red by $ (v,X,Y,k) $ if and only if 
\begin{equation}
\label{eq:linear_comb}
(k+1)\cdot\operatorname{row}_{A(G)}(u)
= \operatorname{row}_{A(G)}(v) + \sum_{x \in X}
\operatorname{row}_{A(G)}(x) - \sum_{y \in Y}
\operatorname{row}_{A(G)}(y).
\end{equation}
\end{obs}

\begin{thm}
\label{thm:null_G_null_AG}
Let $ G $ be a simple graph. Then $ \operatorname{null}(G) =
\operatorname{null}(A(G)) $.
\end{thm}

\begin{proof}
Let $ G $ be a graph with all vertices initially colored white. Suppose that at
some stage the vertices $ u_1,u_2,\dots,u_{q-1} $ have been sequentially
colored red, the remaining vertices colored white, and that each $
\operatorname{row}_{A(G)}(u_i) $ can be expressed as a linear combination of
rows indexed $ W = V(G) \setminus \{u_1,u_2,\dots,u_{q-1}\} $. Suppose that $ v
$ and the vertices of $ X,Y $ are white and $ u_q $ can be colored red by $
(v,X,Y,k) $. We show that $ \operatorname{row}_{A(G)}(u_i) $ for $ i =
1,2,\dots,q $ can each be expressed as a linear combination of rows indexed by
$ W' = W \setminus \{u_q\} $. 
\bigskip

Let $ W' = \{w_1, w_2, \dots, w_\ell \} $. By
\hyperref[eq:linear_comb]{(\ref{eq:linear_comb})}, $ 
\operatorname{row}_{A(G)}(u_q) $ can be expressed as a linear combination of rows
indexed by $ W' $. We know that, $ \operatorname{row}_{A(G)}(u_i)
$ can be expressed as a linear combination of the rows associated with the vertices in $ W = W' \cup \{u_q\}$. 
By substituting the expression for $ \operatorname{row}_{A(G)}(u_q) $ into that for $ \operatorname{row}_{A(G)}(u_i) $, we see that $ \operatorname{row}_{A(G)}(u_i) $ is a linear
combination of rows associated with vertices in $ W' $.
At the conclusion of this process $
\operatorname{rank}(A(G)) \leq n - \operatorname{null}(G)$, so $
\operatorname{null}(G) \leq \operatorname{null}(A(G))  $.
\bigskip

Let $ W $ be a set of linearly independent rows of $ A(G) $ that forms a basis
for the row space of $ A(G) $. Let $ r = |W| $ and let $ v_1,v_2, \dots, v_r $
be the vertices associated with these rows. Then each row not in $ W $, $
\operatorname{row}_{A(G)}(v_j) $ with $ j > r $, can be written as 
\[
\frac{c_1}{d_1} \operatorname{row}_{A(G)}(v_1) +\frac{c_2}{d_2} \operatorname{row}_{A(G)}(v_2) + \cdots + \frac{c_r}{d_r} \operatorname{row}_{A(G)}(v_r)
\]
where $ c_i,d_i \in \mathbb{Z} $ and $ d_i > 0 $ for $ i = 1, \dots , r $.
By letting $ d = \operatorname{lcm}(d_1,d_2, \dots,d_r) $ we can write 
\begin{equation}
d \cdot \operatorname{row}_{A(G)}(v_j) = {c_1}{s_1} \operatorname{row}_{A(G)}(v_1) +{c_2}{s_2} \operatorname{row}_{A(G)}(v_2) + \cdots + {c_r}{s_r} \operatorname{row}_{A(G)}(v_r)
\end{equation}
where $ s_i = d/d_i \in \mathbb{Z} $. Fix $ v_j $ corresponding to a row in $ W
$. Let $ \ell \in \{1,2,\dots,r\} $ such that $ c_\ell s_\ell > 0 $. Let $ X $
be the multiset of vertices consisting of $ c_\ell s_\ell -1 $ copies of
$ v_\ell $ and $ c_i s_i $ copies of $ v_i $ for $ i \neq \ell $ and $ c_i s_i
> 0$ and let $ Y $ be the multiset of vertex consisting of $ c_is_i $ copies of
$ v_i $ for $ c_is_i < 0 $. Then $ v_j $ can be colored red by $ (v_\ell,X,Y,d-1) $. This implies $
\operatorname{null}(G) \geq n - r  \geq n - \operatorname{rank}(A(G)) =
\operatorname{null}(A(G))$. 
\end{proof}

\begin{cor}
\label{cor:null_of_bipartite}
Let $ G $ be a bipartite graph with independent sets $ \mathbf{B}  $ and $
\mathbf{\bar{B}} $ such that $ |\mathbf{B}| = |\mathbf{\bar{B}}| $. Let $ \mathcal{R}
\subseteq \mathbf{B} $ be a red set  such that every vertex in $ \mathcal{R} $
is colored with some $ (v,X,Y,k) $ where $ \{v\} \cup X \cup Y $ contains only
vertices from $ \mathbf{B} $. Then $2|\mathcal{R}| \leq
\operatorname{null}(A(G))$.
\end{cor}

The \textit{Aztec diamond of order $ r $} is a diamond shape configuration of
2r(r+1) unit squares, as illustrated in \hyperref[fig:Aztec_diamond]{Figure \ref{fig:Aztec_diamond}}. The \textit{Aztec diamond graph of order $ r $},
denoted by $ \operatorname{AD}_r $, is the graph such that vertices $ v,u \in
V(\operatorname{AD}_r) $ are adjacent if and only if squares $ v $ and $ u $
share an edge in the Aztec diamond of order $ r $. The vertices of $
\operatorname{AD}_r $ are labeled by ordered pairs $ (i,j) $ where $ 1 \leq i,j
\leq 2r$, $ r+1 \leq i+j \leq 3r+1 $, and $ 0 \leq |j-i| \leq r $.
\bigskip

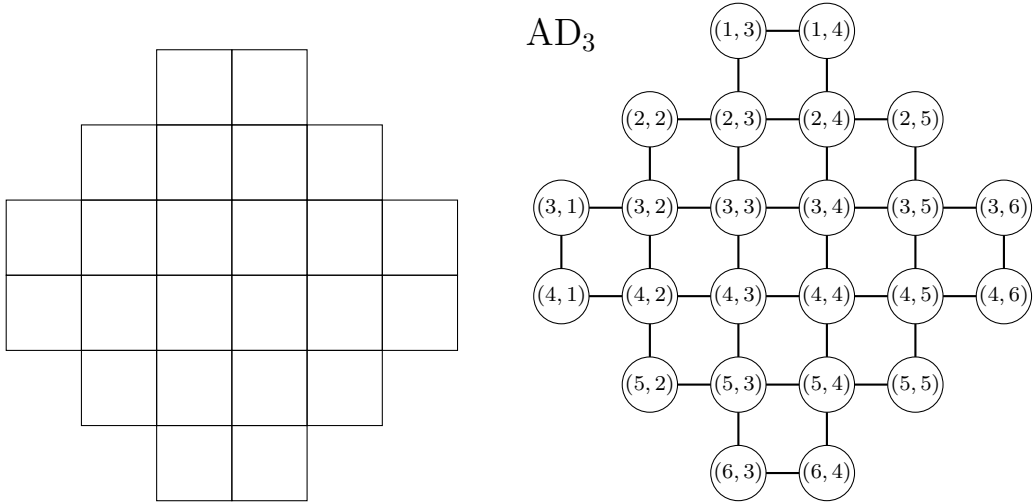
\begin{figure}[h]
\begin{center}
\begin{tikzpicture}
\draw (0,0) rectangle (1,1);
\draw (1,0) rectangle (2,1);

\draw (0,1) rectangle (1,2);
\draw (1,1) rectangle (2,2);
\draw (-1,1) rectangle (0,2);
\draw (2,1) rectangle (3,2);

\draw (0,2) rectangle (1,3);
\draw (1,2) rectangle (2,3);
\draw (-1,2) rectangle (0,3);
\draw (2,2) rectangle (3,3);
\draw (-2,2) rectangle (-1,3);
\draw (3,2) rectangle (4,3);

\draw (0,3) rectangle (1,4);
\draw (1,3) rectangle (2,4);
\draw (-1,3) rectangle (0,4);
\draw (2,3) rectangle (3,4);
\draw (-2,3) rectangle (-1,4);
\draw (3,3) rectangle (4,4);

\draw (0,4) rectangle (1,5);
\draw (1,4) rectangle (2,5);
\draw (-1,4) rectangle (0,5);
\draw (2,4) rectangle (3,5);

\draw (0,5) rectangle (1,6);
\draw (1,5) rectangle (2,6);
\end{tikzpicture}
\hspace{0.4cm}
\begin{tikzpicture}[scale=0.53,font=\scriptsize]
\GraphInit[vstyle=Normal]
\tikzset{VertexStyle/.style = {shape = circle, font=\large}}
\Vertex[L=\hbox{$\operatorname{AD}_3$},x=0cm,y=11cm]{G}
\tikzset{VertexStyle/.style = {shape = circle, draw=black,
minimum size=8pt, inner sep=0.2pt}}
\Vertex[L=\hbox{$\left(1, 3\right)$},x=4.44cm,y=11.1cm]{v0}
\Vertex[L=\hbox{$\left(1, 4\right)$},x=6.66cm,y=11.1cm]{v1}
\Vertex[L=\hbox{$\left(2, 2\right)$},x=2.22cm,y=8.88cm]{v2}
\Vertex[L=\hbox{$\left(2, 3\right)$},x=4.44cm,y=8.88cm]{v3}
\Vertex[L=\hbox{$\left(2, 4\right)$},x=6.66cm,y=8.88cm]{v4}
\Vertex[L=\hbox{$\left(2, 5\right)$},x=8.88cm,y=8.88cm]{v5}
\Vertex[L=\hbox{$\left(3, 1\right)$},x=0.0cm,y=6.66cm]{v6}
\Vertex[L=\hbox{$\left(3, 2\right)$},x=2.22cm,y=6.66cm]{v7}
\Vertex[L=\hbox{$\left(3, 3\right)$},x=4.44cm,y=6.66cm]{v8}
\Vertex[L=\hbox{$\left(3, 4\right)$},x=6.66cm,y=6.66cm]{v9}
\Vertex[L=\hbox{$\left(3, 5\right)$},x=8.88cm,y=6.66cm]{v10}
\Vertex[L=\hbox{$\left(3, 6\right)$},x=11.1cm,y=6.66cm]{v11}
\Vertex[L=\hbox{$\left(4, 1\right)$},x=0.0cm,y=4.44cm]{v12}
\Vertex[L=\hbox{$\left(4, 2\right)$},x=2.22cm,y=4.44cm]{v13}
\Vertex[L=\hbox{$\left(4, 3\right)$},x=4.44cm,y=4.44cm]{v14}
\Vertex[L=\hbox{$\left(4, 4\right)$},x=6.66cm,y=4.44cm]{v15}
\Vertex[L=\hbox{$\left(4, 5\right)$},x=8.88cm,y=4.44cm]{v16}
\Vertex[L=\hbox{$\left(4, 6\right)$},x=11.1cm,y=4.44cm]{v17}
\Vertex[L=\hbox{$\left(5, 2\right)$},x=2.22cm,y=2.22cm]{v18}
\Vertex[L=\hbox{$\left(5, 3\right)$},x=4.44cm,y=2.22cm]{v19}
\Vertex[L=\hbox{$\left(5, 4\right)$},x=6.66cm,y=2.22cm]{v20}
\Vertex[L=\hbox{$\left(5, 5\right)$},x=8.88cm,y=2.22cm]{v21}
\Vertex[L=\hbox{$\left(6, 3\right)$},x=4.44cm,y=0.0cm]{v22}
\Vertex[L=\hbox{$\left(6, 4\right)$},x=6.66cm,y=0.0cm]{v23}
\Edge[](v0)(v1)
\Edge[](v0)(v3)
\Edge[](v1)(v4)
\Edge[](v2)(v3)
\Edge[](v2)(v7)
\Edge[](v3)(v4)
\Edge[](v3)(v8)
\Edge[](v4)(v5)
\Edge[](v4)(v9)
\Edge[](v5)(v10)
\Edge[](v6)(v7)
\Edge[](v6)(v12)
\Edge[](v7)(v8)
\Edge[](v7)(v13)
\Edge[](v8)(v9)
\Edge[](v8)(v14)
\Edge[](v9)(v10)
\Edge[](v9)(v15)
\Edge[](v10)(v11)
\Edge[](v10)(v16)
\Edge[](v11)(v17)
\Edge[](v12)(v13)
\Edge[](v13)(v14)
\Edge[](v13)(v18)
\Edge[](v14)(v15)
\Edge[](v14)(v19)
\Edge[](v15)(v16)
\Edge[](v15)(v20)
\Edge[](v16)(v17)
\Edge[](v16)(v21)
\Edge[](v18)(v19)
\Edge[](v19)(v20)
\Edge[](v19)(v22)
\Edge[](v20)(v21)
\Edge[](v20)(v23)
\Edge[](v22)(v23)
\end{tikzpicture}
\end{center}
\caption{The Aztec diamond of order $ 3 $ and the Aztec diamond graph $
\operatorname{AD}_3 $.}
\label{fig:Aztec_diamond}
\end{figure}

\begin{prop}
\label{prop:zero_forcing_number_aztec_graph}
Let $ G $ be a Aztec diamond graph $ \operatorname{AD}_r $. Then $ \operatorname{Z}(G)
\leq 2r $.
\end{prop}

\begin{proof}
We show that the set $ Z = \{(1,r),(2,r-1),(3,r-2),\dots,(r,1)\} \cup
\{(1,r+1),(2,r+2),(3,r+3),\dots,(r,2r)\}$ is a zero forcing set. For $ i \in
\{1,2,\dots,r\}$ in order $ (i,j) $ can force $ (i+1,j) $ as long as $ (i,j) $
and $ (i+1,j) $ exist.
\end{proof}

\begin{thm}
\label{thm:m_equals_z_aztec_diamond}
Let $ \operatorname{AD}_r $ be a Aztec diamond graph of order $ r $  and $ \mathcal{F} $
be an arbitrary field. Then
\[
\operatorname{M}(\mathcal{F},\operatorname{AD}_r)= \operatorname{Z}(\operatorname{AD}_r)=2r
\]
and field independent minimum rank is established with the universally optimal matrix $ A(G) $.
\end{thm}

\begin{proof}
Let $
D_\ell = \{ (i+\ell,r+2+\ell-i) | 1 \leq i \leq r+1 \} $ for $ 0 \leq \ell \leq
r-1$. Note that the $ D_\ell $ are independent sets and disjoint.
Let $ \mathbf{B} = D_0\cup D_1 \cup D_2 \cup \dots \cup D_{(r-1)} $. We show
that $ r $ vertices of $ \mathbf{B} $ can be colored red by other vertices of $ \mathbf{B}
$. The vertex $ (r+1,1) $ in the set $ D_0 $ can be colored red by $
\big((r,2),\{(i,j) \in D_0 | i < r , j \text{ is even}\},\{(i,j) \in D_0 | i < r , j
\text{ is odd}\},0\big)$. See \hyperref[fig:Aztex_diamond_coloring]{Figure
\ref{fig:Aztex_diamond_coloring}} for an example. Using a similar argument each $ D_\ell $ has a vertex
that can be colored red using only vertices from $ D_\ell $. Since $
\mathbf{B} $ is partitioned into $ r $ sets $ D_\ell $,
a total of $ r $ vertices that can be colored red. By
\hyperref[cor:null_of_bipartite]{Corollary \ref{cor:null_of_bipartite}}, $
2r \leq \operatorname{null}( \operatorname{AD}_r ) $. By
\hyperref[thm:null_G_null_AG]{Theorem \ref{thm:null_G_null_AG}} and
\hyperref[prop:zero_forcing_number_aztec_graph]{Proposition
\ref{prop:zero_forcing_number_aztec_graph}}, $ 2r \leq
\operatorname{null}(A(\operatorname{AD}_r)) \leq
\operatorname{M}(\operatorname{AD}_r) \leq
\operatorname{Z}(\operatorname{AD}_r) \leq 2r $.
\end{proof}

\begin{prop}
\label{prop:m_equals_z_circ_n_divides_8}
Let $ n $ be a multiple of $ 8 $. Then,
\[
\operatorname{Z}(\operatorname{Circ}[n,\{1,\tfrac{n}{2}-1\}]) \leq
\tfrac{n}{2}+2.
\]
\end{prop}

\begin{proof}
Let $ G = \operatorname{Circ}[n,\{1,\tfrac{n}{2}-1\}] $.
Then $ Z = \{ 0,1,2, \dots,\tfrac{n}{2}, n-1 \} $ is a
zero forcing set with forces $ 0 \to n/2 + 1, 1 \to n/2 + 2 , \cdots ,
n/2-3 \to n-2$. This shows that
$ \operatorname{Z}(G)\leq\tfrac{n}{2}+2
$
\end{proof}

\begin{figure}[h]
\begin{center}
\begin{tikzpicture}[scale=0.48,font=\scriptsize]
\GraphInit[vstyle=Normal]
\tikzset{VertexStyle/.style = {shape = circle, font=\large}}
\Vertex[L=\hbox{$\operatorname{AD}_{3_{(3,2)+(1,4)}}$},x=0cm,y=13cm]{G}
\tikzset{VertexStyle/.style = {shape = circle, draw=black,
minimum size=8pt, inner sep=0.2pt}}
\Vertex[L=\hbox{$\left(1, 3\right)$},x=4.44cm,y=11.1cm]{v0}
\Vertex[L=\hbox{$\left(1, 4\right)$},x=6.66cm,y=11.1cm]{v1}
\Vertex[L=\hbox{$\left(2, 4\right)$},x=6.66cm,y=8.88cm]{v4}
\Vertex[L=\hbox{$\left(2, 2\right)$},x=2.22cm,y=8.88cm]{v2}
\Vertex[L=\hbox{$\left(2, 3\right)$},x=4.44cm,y=8.88cm]{v3}
\Vertex[L=\hbox{$\left(2, 5\right)$},x=8.88cm,y=8.88cm]{v5}
\Vertex[L=\hbox{$\left(3, 1\right)$},x=0.0cm,y=6.66cm]{v6}
\Vertex[L=\hbox{$\left(3, 2\right)$},x=2.22cm,y=6.66cm]{v7}
\Vertex[L=\hbox{$\left(3, 3\right)$},x=4.44cm,y=6.66cm]{v8}
\Vertex[L=\hbox{$\left(3, 4\right)$},x=6.66cm,y=6.66cm]{v9}
\Vertex[L=\hbox{$\left(3, 5\right)$},x=8.88cm,y=6.66cm]{v10}
\Vertex[L=\hbox{$\left(3, 6\right)$},x=11.1cm,y=6.66cm]{v11}
\Vertex[L=\hbox{$\left(4, 1\right)$},x=0.0cm,y=4.44cm]{v12}
\Vertex[L=\hbox{$\left(4, 2\right)$},x=2.22cm,y=4.44cm]{v13}
\Vertex[L=\hbox{$\left(4, 3\right)$},x=4.44cm,y=4.44cm]{v14}
\Vertex[L=\hbox{$\left(4, 4\right)$},x=6.66cm,y=4.44cm]{v15}
\Vertex[L=\hbox{$\left(4, 5\right)$},x=8.88cm,y=4.44cm]{v16}
\Vertex[L=\hbox{$\left(4, 6\right)$},x=11.1cm,y=4.44cm]{v17}
\Vertex[L=\hbox{$\left(5, 2\right)$},x=2.22cm,y=2.22cm]{v18}
\Vertex[L=\hbox{$\left(5, 3\right)$},x=4.44cm,y=2.22cm]{v19}
\Vertex[L=\hbox{$\left(5, 4\right)$},x=6.66cm,y=2.22cm]{v20}
\Vertex[L=\hbox{$\left(5, 5\right)$},x=8.88cm,y=2.22cm]{v21}
\Vertex[L=\hbox{$\left(6, 3\right)$},x=4.44cm,y=0.0cm]{v22}
\Vertex[L=\hbox{$\left(6, 4\right)$},x=6.66cm,y=0.0cm]{v23}
\Edge[](v4)(v1)
\Edge[](v0)(v1)
\Edge[](v0)(v7)
\Edge[](v0)(v3)
\Edge[](v7)(v4)
\Edge[](v2)(v3)
\Edge[](v2)(v7)
\Edge[](v3)(v4)
\Edge[](v3)(v8)
\Edge[](v4)(v5)
\Edge[](v4)(v9)
\Edge[](v5)(v10)
\Edge[](v6)(v7)
\Edge[](v6)(v12)
\Edge[](v7)(v8)
\Edge[](v7)(v13)
\Edge[](v8)(v9)
\Edge[](v8)(v14)
\Edge[](v9)(v10)
\Edge[](v9)(v15)
\Edge[](v10)(v11)
\Edge[](v10)(v16)
\Edge[](v11)(v17)
\Edge[](v12)(v13)
\Edge[](v13)(v14)
\Edge[](v13)(v18)
\Edge[](v14)(v15)
\Edge[](v14)(v19)
\Edge[](v15)(v16)
\Edge[](v15)(v20)
\Edge[](v16)(v17)
\Edge[](v16)(v21)
\Edge[](v18)(v19)
\Edge[](v19)(v20)
\Edge[](v19)(v22)
\Edge[](v20)(v21)
\Edge[](v20)(v23)
\Edge[](v22)(v23)
\end{tikzpicture}
\hspace{0.1cm}
\begin{tikzpicture}[scale=0.48,font=\scriptsize]
\GraphInit[vstyle=Normal]
\tikzset{VertexStyle/.style = {shape = circle, font=\large}}
\Vertex[L=\hbox{$\operatorname{AD}_{3_{(3,2)+(1,4)-(2,3)}}$},x=0cm,y=13cm]{G}
\tikzset{VertexStyle/.style = {shape = circle, draw=black,
minimum size=8pt, inner sep=0.2pt}}
\Vertex[L=\hbox{$\left(1, 3\right)$},x=4.44cm,y=11.1cm]{v0}
\Vertex[L=\hbox{$\left(1, 4\right)$},x=6.66cm,y=11.1cm]{v1}
\Vertex[L=\hbox{$\left(2, 2\right)$},x=2.22cm,y=8.88cm]{v2}
\Vertex[L=\hbox{$\left(2, 3\right)$},x=4.44cm,y=8.88cm]{v3}
\Vertex[L=\hbox{$\left(2, 4\right)$},x=6.66cm,y=8.88cm]{v4}
\Vertex[L=\hbox{$\left(2, 5\right)$},x=8.88cm,y=8.88cm]{v5}
\Vertex[L=\hbox{$\left(3, 1\right)$},x=0.0cm,y=6.66cm]{v6}
\Vertex[L=\hbox{$\left(3, 2\right)$},x=2.22cm,y=6.66cm]{v7}
\Vertex[L=\hbox{$\left(3, 3\right)$},x=4.44cm,y=6.66cm]{v8}
\Vertex[L=\hbox{$\left(3, 4\right)$},x=6.66cm,y=6.66cm]{v9}
\Vertex[L=\hbox{$\left(3, 5\right)$},x=8.88cm,y=6.66cm]{v10}
\Vertex[L=\hbox{$\left(3, 6\right)$},x=11.1cm,y=6.66cm]{v11}
\Vertex[L=\hbox{$\left(4, 1\right)$},x=0.0cm,y=4.44cm]{v12}
\Vertex[L=\hbox{$\left(4, 2\right)$},x=2.22cm,y=4.44cm]{v13}
\Vertex[L=\hbox{$\left(4, 3\right)$},x=4.44cm,y=4.44cm]{v14}
\Vertex[L=\hbox{$\left(4, 4\right)$},x=6.66cm,y=4.44cm]{v15}
\Vertex[L=\hbox{$\left(4, 5\right)$},x=8.88cm,y=4.44cm]{v16}
\Vertex[L=\hbox{$\left(4, 6\right)$},x=11.1cm,y=4.44cm]{v17}
\Vertex[L=\hbox{$\left(5, 2\right)$},x=2.22cm,y=2.22cm]{v18}
\Vertex[L=\hbox{$\left(5, 3\right)$},x=4.44cm,y=2.22cm]{v19}
\Vertex[L=\hbox{$\left(5, 4\right)$},x=6.66cm,y=2.22cm]{v20}
\Vertex[L=\hbox{$\left(5, 5\right)$},x=8.88cm,y=2.22cm]{v21}
\Vertex[L=\hbox{$\left(6, 3\right)$},x=4.44cm,y=0.0cm]{v22}
\Vertex[L=\hbox{$\left(6, 4\right)$},x=6.66cm,y=0.0cm]{v23}
%
\Edge[](v0)(v3)
\Edge[](v2)(v3)
\Edge[](v4)(v1)
\Edge[](v0)(v1)
\Edge[](v3)(v4)
\Edge[](v3)(v8)
\Edge[](v4)(v5)
\Edge[](v4)(v9)
\Edge[](v5)(v10)
\Edge[](v6)(v7)
\Edge[](v6)(v12)
\Edge[](v7)(v13)
\Edge[](v8)(v9)
\Edge[](v8)(v14)
\Edge[](v9)(v10)
\Edge[](v9)(v15)
\Edge[](v10)(v11)
\Edge[](v10)(v16)
\Edge[](v11)(v17)
\Edge[](v12)(v13)
\Edge[](v13)(v14)
\Edge[](v13)(v18)
\Edge[](v14)(v15)
\Edge[](v14)(v19)
\Edge[](v15)(v16)
\Edge[](v15)(v20)
\Edge[](v16)(v17)
\Edge[](v16)(v21)
\Edge[](v18)(v19)
\Edge[](v19)(v20)
\Edge[](v19)(v22)
\Edge[](v20)(v21)
\Edge[](v20)(v23)
\Edge[](v22)(v23)
\end{tikzpicture}
\end{center}
\caption{Coloring $ (4,1) $ red with $ ((3,2),\{(1,4)\},\{(2,3)\},0) $ in the Aztec diamond graph $ \operatorname{AD}_3 $.}
\label{fig:Aztex_diamond_coloring}
\end{figure}

A \textit{circulant} graph, denoted by $ \operatorname{Circ}[n,S] $, is a graph
with vertex set $ \{ 0, 1, \dots, n-1 \} \subseteq \mathbb{Z} $ and
a \textit{connection set}  $ S \subseteq \{1,2,\dots, \frac{n}{2}\} \subseteq
\mathbb{Z}$, where the edge set of $ \operatorname{Circ}[n,S] $ is precisely
$ \big\{ \{i,i\pm s \} : s \in S \}\big\} $ with arithmetic performed modulo
$ n $ (see \hyperref[fig:circ_24_13_ex]{Figure \ref{fig:circ_24_13_ex}}). For
any $ a \in [n] $, the graphs $ \operatorname{Circ}[n,S] $ and
$ \operatorname{Circ}[n,aS]  $ are isomorphic whenever $ a $ and $ n $ are
relatively prime. Thus if there exists $ b \in S $ such that $ \gcd(b,n) = 1 $,
then $ 1 \in b^{-1}S $ and $ \operatorname{Circ}[n,S] \cong
\operatorname{Circ}[n,b^{-1}S]$. For simplicity, all circulant graphs
considered here have $ 1 $ in the connection set. 

\begin{figure}[ht]
\centering
\begin{tikzpicture}[scale=0.35,font=\scriptsize]
\GraphInit[vstyle=Normal]
\tikzset{VertexStyle/.style = {shape = circle, draw=black, minimum size=14pt, inner sep=0.5pt}}
\Vertex[L=\hbox{$0$},x=11.1cm,y=5.55cm]{v0}
\Vertex[L=\hbox{$1$},x=9.4744cm,y=9.4744cm]{v1}
\Vertex[L=\hbox{$2$},x=5.55cm,y=11.1cm]{v2}
\Vertex[L=\hbox{$3$},x=1.6256cm,y=9.4744cm]{v3}
\Vertex[L=\hbox{$4$},x=0.0cm,y=5.55cm]{v4}
\Vertex[L=\hbox{$5$},x=1.6256cm,y=1.6256cm]{v5}
\Vertex[L=\hbox{$6$},x=5.55cm,y=0.0cm]{v6}
\Vertex[L=\hbox{$7$},x=9.4744cm,y=1.6256cm]{v7}
\Edge[](v0)(v1)
\Edge[](v0)(v2)
\Edge[](v0)(v6)
\Edge[](v0)(v7)
\Edge[](v1)(v2)
\Edge[](v1)(v3)
\Edge[](v1)(v7)
\Edge[](v2)(v3)
\Edge[](v2)(v4)
\Edge[](v3)(v4)
\Edge[](v3)(v5)
\Edge[](v4)(v5)
\Edge[](v4)(v6)
\Edge[](v5)(v6)
\Edge[](v5)(v7)
\Edge[](v6)(v7)
\end{tikzpicture}
\hspace{2cm}
\begin{tikzpicture}[scale=0.35,font=\scriptsize]
\GraphInit[vstyle=Normal]
\tikzset{VertexStyle/.style = {shape = circle, draw=black, minimum size=14pt, inner sep=0.5pt}}
\Vertex[L=\hbox{$0$},x=11.1cm,y=5.55cm]{v0}
\Vertex[L=\hbox{$1$},x=9.4744cm,y=9.4744cm]{v1}
\Vertex[L=\hbox{$2$},x=5.55cm,y=11.1cm]{v2}
\Vertex[L=\hbox{$3$},x=1.6256cm,y=9.4744cm]{v3}
\Vertex[L=\hbox{$4$},x=0.0cm,y=5.55cm]{v4}
\Vertex[L=\hbox{$5$},x=1.6256cm,y=1.6256cm]{v5}
\Vertex[L=\hbox{$6$},x=5.55cm,y=0.0cm]{v6}
\Vertex[L=\hbox{$7$},x=9.4744cm,y=1.6256cm]{v7}
\Edge[](v0)(v1)
\Edge[](v0)(v3)
\Edge[](v0)(v5)
\Edge[](v0)(v7)
\Edge[](v1)(v2)
\Edge[](v1)(v4)
\Edge[](v1)(v6)
\Edge[](v2)(v3)
\Edge[](v2)(v5)
\Edge[](v2)(v7)
\Edge[](v3)(v4)
\Edge[](v3)(v6)
\Edge[](v4)(v5)
\Edge[](v4)(v7)
\Edge[](v5)(v6)
\Edge[](v6)(v7)
\end{tikzpicture}
\caption{ The circulants $ \operatorname{Circ}[8,\{1,2\}] $ and $ \operatorname{Circ}[8,\{1,3\}] $.}
\label{fig:circ_24_13_ex}
\end{figure}

\begin{thm}
\label{thm:m_equals_z_circ_n_divides_8}
Let $ n $ be a multiple of $ 8 $. Then,
\[
\operatorname{M}(\operatorname{Circ}[n,\{1,\tfrac{n}{2}-1\}])=
\operatorname{Z}(\operatorname{Circ}[n,\{1,\tfrac{n}{2}-1\}])=\tfrac{n}{2}+2
\]
and field independent minimum rank  is established  with the universally optimal matrix $ A(G) $.
\bigskip

\end{thm}

\begin{proof}
First note that $ G $ is bipartite with partite set $ \mathbf{B} = \{ 2k \, |
\, 0 \leq k \leq \tfrac{n}{2} - 1 \}  $ and $ \mathbf{\bar{B}} = \{2k+1 \, | \, 0
\leq k \leq \tfrac{n}{2} - 1\} $.  We show that $ \frac{n}{4} + 1 $
vertices from $ \mathbf{B} $ can be colored red using only white vertices of $
\mathbf{B} $.  Note that for every vertex $ v $ in $ \{0,1,2, \dots ,
\tfrac{n}{2}-1\} $, $ v $ is adjacent to $ v+1, v-1, v+\tfrac{n}{2}-1,
v+\tfrac{n}{2}+1 $, and $ v + \tfrac{n}{2}$ is adjacent $ v+\tfrac{n}{2}+1,
v+\tfrac{n}{2}-1, v+\tfrac{n}{2}+\tfrac{n}{2}-1 \equiv v - 1 \bmod n,
v+\tfrac{n}{2}+\tfrac{n}{2}+1 \equiv v+1 \bmod n$. Hence, $ N_G(v) =
N_G(v+\tfrac{n}{2})$ and $ v $ can be colored red by $
(v+\tfrac{n}{2},\emptyset, \emptyset,0)$ where $ v \in \{ 0,2,4,\dots,
\tfrac{n}{2}-2 \}$. This shows that $ \tfrac{n}{4} $ vertices from $ \mathbf{B}
$ can be colored red. The vertex $ \tfrac{n}{2} $ can be colored red by $
(\tfrac{n}{2}+2,\{2i \, : \, 2 | i \text{ and } \tfrac{n}{2}+2 < 2i \leq n-1\},\{2i
\, : \, 2 \nmid i \text{ and } \tfrac{n}{2}+2 < 2i \leq n-1 \},0) $. 
Hence, the vertices of $ \{0,2,4,\dots,\tfrac{n}{2} \} $ can be colored red
with the vertices $ \mathbf{B} \setminus \{0,2,4,\dots,\tfrac{n}{2} \} $. By
\hyperref[cor:null_of_bipartite]{Corollary \ref{cor:null_of_bipartite}}, $
2(\tfrac{n}{4} + 1) = \tfrac{n}{2} + 2 \leq \operatorname{null}(G)  $. So by \hyperref[thm:null_G_null_AG]{Theorem \ref{thm:null_G_null_AG}} and by
\hyperref[prop:m_equals_z_circ_n_divides_8]{Proposition
\ref{prop:m_equals_z_circ_n_divides_8}} 
\[
\tfrac{n}{2} + 2 \leq  \operatorname{null}(A(G)) \leq \operatorname{M}(G) \leq \operatorname{Z}(G) \leq
\tfrac{n}{2} + 2.
\]
\end{proof}
\section{An application of the strong arnold property}
\label{sec:strong_arnold_property}
In this section, we use the Colin de Verdi\`ere type parameter $ \xi $ to show that 
the maximum nullity and zero forcing number of various families of graphs are
equal. 
\bigskip

A matrix $ A \in \mathcal{S}(G) $ is said to have the \textit{Strong Arnold
Property} (SAP) if there does not exist a nonzero symmetric matrix $ X $ having
the following three properties: (1) $ AX = 0 $, (2) $ A \circ X = 0 $, (3) $ I
\circ X = 0 $ where $ \circ $ is the Hadamard (entrywise) product. The Colin de
Verdi\`ere type parameter associated with the maximum nullity is \[ \xi(G) =
\max\{ \operatorname{null}(A)\, |\, A \in \mathcal{S}(G) \text{ and } A \text{
has the SAP}\}.  \] Clearly $ \xi(G) \leq \operatorname{M}(G) \leq
\operatorname{Z}(G) $ for all graphs $ G $. The parameter $ \xi $ was
introduced in 2005 in \cite{MR2181887} to gain more insight on the minimum rank
of a graph.

\begin{ex}
\label{ex:cycle_8_path_3}
By using SageMath (see \cite{SWforSAP}), $ A(C_8 \square P_3) $ has the SAP and $
\operatorname{null}(A(C_8 \square P_3)) = 6 $. By
\hyperref[thm:M_equal_cartesian_product]{Theorem
\ref{thm:M_equal_cartesian_product}},  $ \operatorname{M}(C_8 \square P_3) =
\operatorname{Z}(C_8 \square P_3) = 6 $. Therefore, $ \xi(C_8
\square P_3) =  \operatorname{M}(C_8 \square P_3) = \operatorname{Z}(C_8 \square P_3) = 6$.
\end{ex}

An edge \textit{contraction} of a graph $ G $ is defined to be a deletion of two
adjacent vertices $ v_1 $ and $ v_2 $ and an insertion of a vertex $ u $ such
that $ uv \in E(G) $ if and only if $ vv_1 \in E(G) $ or $ vv_2 \in E(G) $. A
graph $ H $ is a \textit{minor} of a graph $ G $ if $ H $ can be constructed
from $ G $ by performing edge deletions, vertex deletions, and/or contractions.
We write $ H \preceq G $ when $ H $ is a minor of $ G $. Note that $ G \preceq
G' \preceq G''$ implies $ G \preceq G'' $.

\begin{obs}
\label{obs:cartesian_product_minor}
Let $ 3 \leq k \leq n $ and $ 1 \leq r \leq t $. Then $ C_k \square P_r \preceq
C_n \square P_t $.
\end{obs}

\begin{thm}{\rm\cite[Corollary 2.5]{MR2181887}}
\label{thm:minor_monotone}
If $ H $ is a minor of $ G $ then $ \xi(H) \leq \xi(G) $.
\end{thm}

\begin{defn}
Let $ H  $ be a minor of $ G $. We say that $ H $ is a \textit{zero forcing minor}
of $ G $ if $ \operatorname{Z}(G) \leq \operatorname{Z}(H) $.
\end{defn}

\begin{thm}
\label{thm:m_equal_z_zeroforcing_minor_family}
Let $H$ be a zero forcing minor of $ G $ such that $
\xi(H) = \operatorname{Z}(H) $. Then $ \xi(G) =
\operatorname{M}(G) = \operatorname{Z}(G) = \operatorname{Z}(H) $.
\end{thm}

\begin{proof}
\label{pf:m_equal_z_zeroforcing_minor_family}
Given that $ H $ is a zero forcing minor, 
$
\operatorname{Z}(G) \leq \operatorname{Z}(H).
$
By \hyperref[thm:minor_monotone]{Theorem \ref{thm:minor_monotone}},
$\xi(H) \leq \xi(G)$ and it follows that 
\[
\operatorname{Z}(H) = \xi(H)  \leq \xi(G) \leq \operatorname{M}(G) \leq
\operatorname{Z}(G) \leq \operatorname{Z}(H). 
\]
Thus the parameters $\xi(G), \operatorname{M}(G),  \operatorname{Z}(G)  $ are
equal to $ \operatorname{Z}(H) $.
\end{proof}

\begin{cor}
\label{cor:zi_equal_Z_cartesian_product_divisor}
Let $ G = C_n \square P_3 $ such that $ 8 \leq n $. Then 
\[
\xi(G) = \operatorname{M}(\mathcal{F},G)=
\operatorname{Z}(G) = 6.
\]
\end{cor}

A $ k $-\textit{subdivision} of an edge, say $ uv $, is an operation on a graph in which
edge $ uv $ is deleted, vertices $ v_1,v_2, \dots, v_k $ and edges $ uv_1,
v_1v_2, v_2v_3, \dots, v_kv $ are added. We say the edge $ uv $ has been $ k
$-subdivided. Whenever $ k =  1 $ we simply say that the edge $ uv $ has been
subdivided. A $ k $\textit{-subdivision edge insertion} on the edges $ uv $ and $
wx $ is an operation on a graph in which edges $ uv $ and $ wx $ are $ k
$-subdivided adding vertices $v_1,v_2, \dots, v_k$ and $ x_1,x_2,\dots,x_k $,
respectively, and edges $ v_1x_1, v_2x_2, \dots, v_kx_k $ are added. The
\textit{cube graph} $ Q_3 $ can be described by an 8 - cycle containing a labeled
vertex set $ \{0,1, \dots, 7 \} $ and added edges $ \{
\{0,5\}, \{1,4\}, \{2,7\}, \{3,6\} \} $ as shown in \hyperref[fig:bidiakis_cubes]{Figure
\ref{fig:bidiakis_cubes}}.

\begin{prop}{\rm\cite[Lemma 8]{MR2419952}}
\label{prop:cube_graph_has_sap}
For the cube graph, $ \xi(Q_3) = 4 = \operatorname{M}(G) = \operatorname{Z}(G)$.
\end{prop}

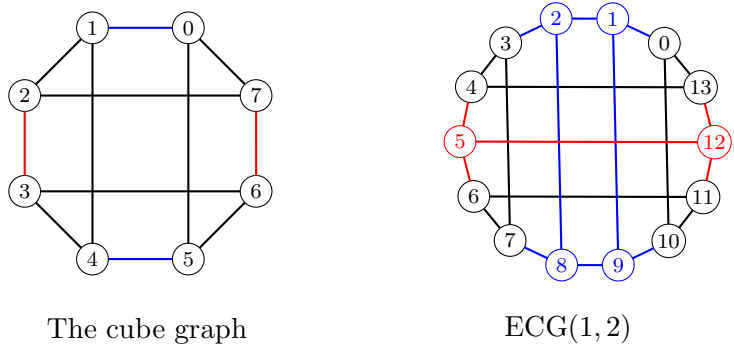
\begin{figure}[t]
\begin{tikzpicture}[scale=0.3, rotate=-22.5, font=\scriptsize]
\GraphInit[vstyle=Normal]
\tikzset{VertexStyle/.style = {shape = circle, draw=black,minimum size=12pt,
inner sep=1pt }} \Vertex[L=\hbox{$0$},x=5.55cm,y=11.1cm]{v0}
\Vertex[L=\hbox{$1$},x=1.6256cm,y=9.4744cm]{v1}
\Vertex[L=\hbox{$2$},x=0.0cm,y=5.55cm]{v2}
\Vertex[L=\hbox{$3$},x=1.6256cm,y=1.6256cm]{v3}
\Vertex[L=\hbox{$4$},x=5.55cm,y=0.0cm]{v4}
\Vertex[L=\hbox{$5$},x=9.4744cm,y=1.6256cm]{v5}
\Vertex[L=\hbox{$6$},x=11.1cm,y=5.55cm]{v6}
\Vertex[L=\hbox{$7$},x=9.4744cm,y=9.4744cm]{v7}
\draw[minimum size=10pt,font=\small](9cm,-2.0cm) node {The cube graph};
\Edge[](v0)(v5) \Edge[](v0)(v7) \Edge[](v1)(v2) \Edge[](v1)(v4) \Edge[](v2)(v7)
\Edge[](v3)(v4) \Edge[](v3)(v6) \Edge[](v5)(v6) \tikzset{EdgeStyle/.style =
{color=blue}} \Edge[](v4)(v5) \Edge[](v0)(v1) \tikzset{EdgeStyle/.style =
{color=red}} \Edge[](v6)(v7) \Edge[](v2)(v3)
\end{tikzpicture} \hspace{2cm} \begin{tikzpicture}[scale=0.3, rotate=-38.0,
font=\scriptsize] \GraphInit[vstyle=Normal]
\tikzset{VertexStyle/.style = {shape = circle, draw=black,minimum size=12pt,
inner sep=1pt }} \Vertex[L=\hbox{$0$},x=5.55cm,y=11.1cm]{v0}
\Vertex[L=\hbox{$3$},x=0.0cm,y=6.785cm]{v3}
\Vertex[L=\hbox{$4$},x=0.0cm,y=4.315cm]{v4}
\Vertex[L=\hbox{$6$},x=3.08cm,y=0.5496cm]{v6}
\Vertex[L=\hbox{$7$},x=5.55cm,y=0.0cm]{v7}
\Vertex[L=\hbox{$10$},x=11.1cm,y=4.315cm]{v10}
\Vertex[L=\hbox{$11$},x=11.1cm,y=6.785cm]{v11}
\Vertex[L=\hbox{$13$},x=8.02cm,y=10.5504cm]{v13} \tikzset{VertexStyle/.style =
{shape = circle, draw=black, minimum size=12pt, inner sep=1pt,color=blue}}
\Vertex[L=\hbox{$1$},x=3.08cm,y=10.5504cm]{v1}
\Vertex[L=\hbox{$2$},x=1.0992cm,y=9.0104cm]{v2}
\Vertex[L=\hbox{$8$},x=8.02cm,y=0.5496cm]{v8}
\Vertex[L=\hbox{$9$},x=10.0008cm,y=2.0896cm]{v9} \tikzset{VertexStyle/.style =
{shape = circle, draw=black, minimum size=12pt, inner sep=1pt,color=red}}
\Vertex[L=\hbox{$5$},x=1.0992cm,y=2.0896cm]{v5}
\Vertex[L=\hbox{$12$},x=10.0008cm,y=9.0104cm]{v12}
\draw[minimum size=10pt,font=\small](10cm,-1.5cm) node {
$\operatorname{ECG}(1,2)$}; \Edge[](v0)(v10) \Edge[](v0)(v13) \Edge[](v3)(v4)
\Edge[](v3)(v7) \Edge[](v4)(v13) \Edge[](v6)(v7) \Edge[](v6)(v11)
\Edge[](v10)(v11) \tikzset{EdgeStyle/.style = {color=blue}} \Edge[](v1)(v9)
\Edge[](v2)(v8) \Edge[](v0)(v1) \Edge[](v1)(v2) \Edge[](v2)(v3) \Edge[](v7)(v8)
\Edge[](v8)(v9) \Edge[](v9)(v10) \tikzset{EdgeStyle/.style = {color=red}}
\Edge[](v4)(v5) \Edge[](v5)(v6) \Edge[](v11)(v12) \Edge[](v12)(v13)
\Edge[](v5)(v12)
\end{tikzpicture} 
\centering
\caption{Applying two vertical and one horizontal subdivision edge insertion on the
cube graph gives $ \operatorname{ECG}(1,2) $.}
\label{fig:bidiakis_cubes}
\end{figure}

\begin{defn}(Extended cube graph)
\label{defn:extended_general_bidiakis_cube}
A \textit{vertical $ k-$subdivision edge insertion} on the cube graph is a $ k-
$subdivision edge insertion on the edges $ \{0,1\} $ and $ \{4,5\} $. A
\textit{horizontal $ k-$subdivision edge insertion} on the cube graph is a $ k-
$subdivision edge insertion on the edges $ \{2,3\} $ and $ \{6,7\} $, with the numbering
as in \hyperref[fig:bidiakis_cubes]{Figure \ref{fig:bidiakis_cubes}}.  An
\textit{extended cube graph}, denoted by $ \operatorname{ECG}(t,k) $, is the
cube graph with a horizontal $ t- $subdivision edge insertion, a vertical $ k-
$subdivision edge insertion, and a relabeling around the cycle containing
vertex set $ \{ 0,1, \dots, 7+2(t+k) \} $. 
\end{defn} 

\hyperref[fig:bidiakis_cubes]{Figure \ref{fig:bidiakis_cubes}} shows $
\operatorname{ECG}(1,2) $. Notice that $ \operatorname{ECG}(t,k) $ isomorphic
to the graph $ \operatorname{ECG}(k,t) $. For
simplicity we consider the extended cube graphs with $ t \leq k $.
The graph $ \operatorname{ECG}(1,1) $ is called the
Bidiakis cube.  It was shown in {\rm\cite[Proposition 5.1]{MR3760976}} that the
maximum nullity and zero forcing number of the Bidiakis cube are the same, motivating the creation of the extended cube graphs.

Observe that in $ \operatorname{ECG}(t,k) $, as we draw it, the top
endpoints of the vertical edges are $ 0, \dots, k+1 $, the left endpoints of
the horizontal edges are $ k+2, \dots, t+k+3 $, the lower endpoints of the
vertical edges are $ t+k+4, \dots, t+2k+5 = n-t-3 $, and the right endpoints of the
horizontal edges are $ t+2k+6,\dots,2t+2k+7=n-1 $.

\begin{obs}
\label{obs:zero_forcing_insertion_minor}
Let $ G $ be a graph constructed from the graph $ H $ by performing a 
subdivision edge insertion. Then $ H \preceq G$.
\end{obs}

\begin{prop}
\label{prop:extended_bidiakis_zero_forcing_number}
Let $ G $ be an extended cube graph $ \operatorname{ECG}(t,k) $. Then $ \operatorname{Z}(G) \leq 4  $.
\end{prop}

\begin{proof}
\label{proof:zero_forcing_number_general_bidiakis_cube}
Let $ n $ be the number of vertices of $ G $ and let $ r = n - t - 3$. The
set $ \{ 0 ,r , r+1,  n-1 \} $ is a zero forcing set with simultaneous forces

\begin{center}
\begin{tikzpicture}
\matrix[row sep=1mm,column sep=1mm] {
\node {$0$}; & \node {$\to$}; & \node {$1$}; & \node {$\to$}; & \node {$2$}; & \node {$\to$}; &
\node {$\cdots$}; & \node {$\to$}; & \node
{$k+2$}; \\
\node {$r$}; & \node {$\to$}; & \node {$r-1$}; & \node {$\to$}; & \node
{$r-2$}; & \node
{$\to$}; & \node {$\cdots$}; & \node {$\to$}; & \node {$r-(k+2) = k+t+3.$}; \\
};
\end{tikzpicture}
\end{center}

These forcing sequences run simultaneously in parallel, i.e., $ 0 \to 1 $ and $ r
\to r-1$ are simultaneous, etc. After the above forces are completed, the
following forces run in parallel.

\begin{center}
\begin{tikzpicture}
\matrix[row sep=1mm,column sep=1mm] {
\node {$(k+2)$}; & \node {$\to$}; & \node {$(k+2)+1$}; & \node {$\to$}; & \node
{$(k+2)+3$}; & \node {$\to$}; &
\node {$\cdots$}; & \node {$\to$}; & \node {$(k+2)+t$};\\
\node {$(n-1)$}; & \node {$\to$}; & \node {$(n-1)-1$}; & \node {$\to$}; & \node
{$(n-1)-2$}; & \node {$\to$}; & \node {$\cdots$}; & \node {$\to$}; & \node
{$(n-1)-t = 2k+t+6$}; \\
};
\end{tikzpicture}
\qedhere
\end{center}

\end{proof}

\begin{cor}
\label{cor:extended_bidiakis_max_nullity_zero_forcing}
Let $ G $ be the extended cube graph $ \operatorname{ECG}(t,k)$. Then 
\[
\xi(G) = \operatorname{M}(G) = \operatorname{Z}(G) = 4.
\] 
\end{cor}

\begin{proof}
\label{pf:extended_bidiakis_max_nullity_zero_forcing}
Let $ H $ be the cube graph. By \hyperref[prop:cube_graph_has_sap]{Proposition
\ref{prop:cube_graph_has_sap}}, $ \xi(H) = \operatorname{Z(H)=4} $. By
\hyperref[thm:m_equal_z_zeroforcing_minor_family]{Theorem
\ref{thm:m_equal_z_zeroforcing_minor_family}}, $ H $ is a zero forcing minor
of $ G $. Thus $ \xi(G) = \operatorname{M}(G) = \operatorname{Z}(G)=4 $ by
\hyperref[thm:m_equal_z_zeroforcing_minor_family]{Theorem
\ref{thm:m_equal_z_zeroforcing_minor_family}}.
\end{proof}

\begin{obs}
\label{obs:first_iso_2k}
For positive integer $ k $, the circulant $ \operatorname{Circ}[4k,\{1,3,\dots,2k-1\}] =
K_{2k,2k} $ and the circulant $ \operatorname{Circ}[4k+2,\{1,3,\dots,2k+1\}] =
K_{2k+1,2k+1}.$
\end{obs}

\hyperref[prop:zero_forcing_number_non_consecutive_circulants]{Proposition
\ref{prop:zero_forcing_number_non_consecutive_circulants}} and 
\hyperref[thm:M_equal_Z_consecutive_circulants]{Theorem
\ref{thm:M_equal_Z_consecutive_circulants}} below  were found by several groups in
2009 and 2010 but not published. Some of these results were also published in
\cite{MR3440139}. We state these results and give formal proofs of
the results for clarity.

\begin{prop}{\rm\cite[Proposition 2.1]{MRcirc}}
\label{prop:zero_forcing_number_non_consecutive_circulants}
Let $ G $ be a circulant graph $ \operatorname{Circ}[n , S] $ and let $ m =
\max\{i | i \in S\} $. Then $ \operatorname{Z}(G) \leq 2m $.
\end{prop}

\begin{proof}
\label{pf:zero_forcing_number_non_consecutive_circulants}
We will show that $ Z = \{ 0 , 1 , \dots, {2m-1} \} $ is a zero forcing set.
Suppose $ s \in S $ and $ s \neq m $. Then $ 1 \leq s < m$ and it follows that
$ {m \pm s }  \in {Z} $. If $ s = m $, then $ m - s = m - m = 0 $ which implies $ {m-s}
\in {Z} $. This shows that all neighbors of $ {m} $ except for $ 2m $ are in $
Z $; clearly $ m \in Z $. Hence $ {m} $ can force $ {2m} $. Using a similar
argument $ {m + i} $ forces $ {2m + i}$ for $ i \in \{1, 2, \dots, n -2m -1 \}
$. A forcing sequence is listed as 
\begin{center}
\resizebox{16cm}{1cm}{
\begin{tikzpicture}
\matrix[row sep=1mm,column sep=1mm] {
\node {${m} \to {2m}, \, {m+1} \to {2m+1}, \, \dots \, , \,{m+(n-2m-1) = n-m-1} \to
{2m+(n-2m-1) = n-1}. $\qedhere};\\
};
\end{tikzpicture}
}
\end{center}
\end{proof}

\begin{obs}
\label{obs:circulant_minor}
Let $ n $ be a multiple of $ k $, $ G = C_{n/k} \, \square \, P_k$, and $ H =
\operatorname{Circ}[n,\{1,k\}] $. Then $ G \preceq H $. This is illustrated in
\hyperref[fig:cycle_8_path_3]{Figure \ref{fig:cycle_8_path_3}}.
\end{obs}

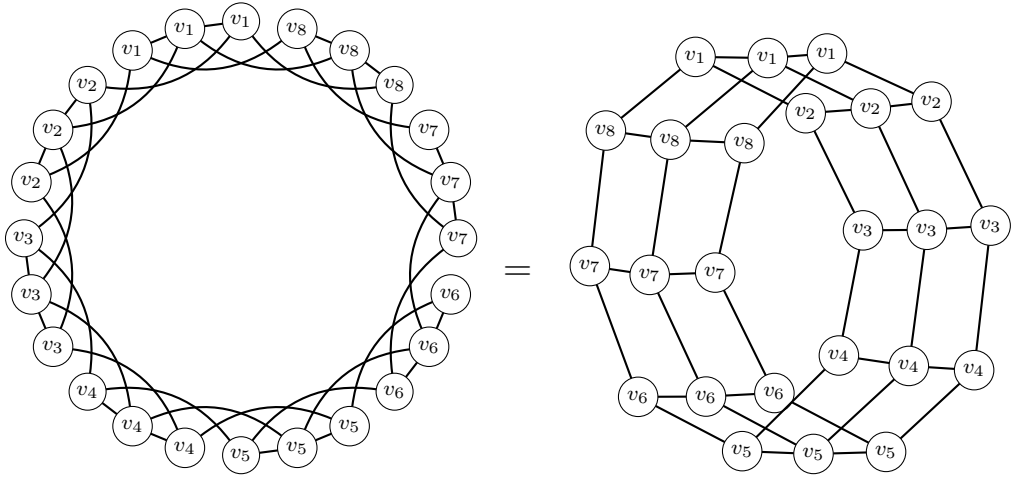
\begin{figure}[ht]
\centering
\begin{tikzpicture}[scale=0.52,font=\scriptsize]
\GraphInit[vstyle=Normal]
%
\tikzset{VertexStyle/.style = {shape = circle, draw=black,
minimum size=14pt, inner sep=0.5pt,color=black}}
\Vertex[L=\hbox{$v_{1}$},x=5.55cm,y=11.1cm]{v0}
\Vertex[L=\hbox{$v_{2}$},x=1.6256cm,y=9.4744cm]{v3}
\Vertex[L=\hbox{$v_{3}$},x=0.0cm,y=5.55cm]{v6}
\Vertex[L=\hbox{$v_{4}$},x=1.6256cm,y=1.6256cm]{v9}
\Vertex[L=\hbox{$v_{5}$},x=5.55cm,y=0.0cm]{v12}
\Vertex[L=\hbox{$v_{6}$},x=9.4744cm,y=1.6256cm]{v15}
\Vertex[L=\hbox{$v_{7}$},x=11.1cm,y=5.55cm]{v18}
\Vertex[L=\hbox{$v_{8}$},x=9.4744cm,y=9.4744cm]{v21}
\tikzset{VertexStyle/.style = {shape = circle, draw=black,
minimum size=0.1pt, inner sep=2pt,color=black}}
\Vertex[L=\hbox{$v_{1}$},x=4.1136cm,y=10.9109cm]{v1}
\Vertex[L=\hbox{$v_{2}$},x=0.7436cm,y=8.325cm]{v4}
\Vertex[L=\hbox{$v_{3}$},x=0.1891cm,y=4.1136cm]{v7}
\Vertex[L=\hbox{$v_{4}$},x=2.775cm,y=0.7436cm]{v10}
\Vertex[L=\hbox{$v_{5}$},x=6.9864cm,y=0.1891cm]{v13}
\Vertex[L=\hbox{$v_{6}$},x=10.3564cm,y=2.775cm]{v16}
\Vertex[L=\hbox{$v_{7}$},x=10.9109cm,y=6.9864cm]{v19}
\Vertex[L=\hbox{$v_{8}$},x=8.325cm,y=10.3564cm]{v22}
\tikzset{VertexStyle/.style = {shape = circle, draw=black,
minimum size=0.1pt, inner sep=2pt,color=black}}
\Vertex[L=\hbox{$v_{1}$},x=2.775cm,y=10.3564cm]{v2}
\Vertex[L=\hbox{$v_{2}$},x=0.1891cm,y=6.9864cm]{v5}
\Vertex[L=\hbox{$v_{3}$},x=0.7436cm,y=2.775cm]{v8}
\Vertex[L=\hbox{$v_{4}$},x=4.1136cm,y=0.1891cm]{v11}
\Vertex[L=\hbox{$v_{5}$},x=8.325cm,y=0.7436cm]{v14}
\Vertex[L=\hbox{$v_{6}$},x=10.9109cm,y=4.1136cm]{v17}
\Vertex[L=\hbox{$v_{7}$},x=10.3564cm,y=8.325cm]{v20}
\Vertex[L=\hbox{$v_{8}$},x=6.9864cm,y=10.9109cm]{v23}
%
\Edge[](v0)(v1)
\Edge[](v1)(v2)
\Edge[](v3)(v4)
\Edge[](v4)(v5)
\Edge[](v6)(v7)
\Edge[](v7)(v8)
\Edge[](v9)(v10)
\Edge[](v10)(v11)
\Edge[](v12)(v13)
\Edge[](v13)(v14)
\Edge[](v15)(v16)
\Edge[](v16)(v17)
\Edge[](v18)(v19)
\Edge[](v19)(v20)
\Edge[](v21)(v22)
\Edge[](v22)(v23)
\tikzset{EdgeStyle/.style = {bend left, color=black}}
\Edge[](v0)(v3)
\Edge[](v3)(v6)
\Edge[](v6)(v9)
\Edge[](v9)(v12)
\Edge[](v12)(v15)
\Edge[](v15)(v18)
\Edge[](v18)(v21)
\Edge[](v21)(v0)
\tikzset{EdgeStyle/.style = {bend left, color=black}}
\Edge[](v1)(v4)
\Edge[](v4)(v7)
\Edge[](v7)(v10)
\Edge[](v10)(v13)
\Edge[](v13)(v16)
\Edge[](v16)(v19)
\Edge[](v19)(v22)
\Edge[](v22)(v1)
\tikzset{EdgeStyle/.style = {bend left, color=black}}
\Edge[](v2)(v5)
\Edge[](v5)(v8)
\Edge[](v8)(v11)
\Edge[](v11)(v14)
\Edge[](v14)(v17)
\Edge[](v20)(v23)
\Edge[](v23)(v2)
\end{tikzpicture}
\begin{tikzpicture}[scale=0.48,font=\scriptsize]
\GraphInit[vstyle=Normal]
\tikzset{VertexStyle/.style = {shape = circle, font=\large}}
\Vertex[L=\hbox{$=$},x=-2cm,y=5cm]{v29}
\tikzset{VertexStyle/.style = {shape = circle, draw=black,
minimum size=0.1pt, inner sep=2pt,color=black}}
\Vertex[L=\hbox{$v_1$},x=6.5689cm,y=11.1cm]{v6}
\Vertex[L=\hbox{$v_2$},x=9.4634cm,y=9.7654cm]{v9}
\Vertex[L=\hbox{$v_3$},x=11.1cm,y=6.3287cm]{v12}
\Vertex[L=\hbox{$v_4$},x=10.6428cm,y=2.3224cm]{v15}
\Vertex[L=\hbox{$v_5$},x=8.2041cm,y=0.0566cm]{v18}
\Vertex[L=\hbox{$v_6$},x=5.1111cm,y=1.689cm]{v21}
\Vertex[L=\hbox{$v_7$},x=3.472cm,y=5.0271cm]{v0}
\Vertex[L=\hbox{$v_8$},x=4.2796cm,y=8.6147cm]{v3}
\tikzset{VertexStyle/.style = {shape = circle, draw=black,
minimum size=0.1pt, inner sep=2pt,color=black}}
\Vertex[L=\hbox{$v_1$},x=4.929cm,y=10.9745cm]{v7}
\Vertex[L=\hbox{$v_2$},x=7.7808cm,y=9.5752cm]{v10}
\Vertex[L=\hbox{$v_3$},x=9.3342cm,y=6.2131cm]{v13}
\Vertex[L=\hbox{$v_4$},x=8.8344cm,y=2.4397cm]{v16}
\Vertex[L=\hbox{$v_5$},x=6.194cm,y=0.0cm]{v19}
\Vertex[L=\hbox{$v_6$},x=3.2166cm,y=1.5907cm]{v22}
\Vertex[L=\hbox{$v_7$},x=1.6574cm,y=4.9626cm]{v1}
\Vertex[L=\hbox{$v_8$},x=2.2328cm,y=8.7113cm]{v4}
\tikzset{VertexStyle/.style = {shape = circle, draw=black,
minimum size=0.1pt, inner sep=2pt,color=black}}
\Vertex[L=\hbox{$v_1$},x=2.927cm,y=11.0059cm]{v8}
\Vertex[L=\hbox{$v_2$},x=5.9784cm,y=9.4441cm]{v11}
\Vertex[L=\hbox{$v_3$},x=7.5662cm,y=6.1977cm]{v14}
\Vertex[L=\hbox{$v_4$},x=6.8974cm,y=2.7286cm]{v17}
\Vertex[L=\hbox{$v_5$},x=4.2191cm,y=0.0774cm]{v20}
\Vertex[L=\hbox{$v_6$},x=1.3381cm,y=1.5834cm]{v23}
\Vertex[L=\hbox{$v_7$},x=0.0cm,y=5.2073cm]{v2}
\Vertex[L=\hbox{$v_8$},x=0.4447cm,y=8.9643cm]{v5}
\Edge[](v0)(v1)
\Edge[](v1)(v2)
\Edge[](v3)(v4)
\Edge[](v4)(v5)
\Edge[](v6)(v7)
\Edge[](v7)(v8)
\Edge[](v9)(v10)
\Edge[](v10)(v11)
\Edge[](v12)(v13)
\Edge[](v13)(v14)
\Edge[](v15)(v16)
\Edge[](v16)(v17)
\Edge[](v18)(v19)
\Edge[](v19)(v20)
\Edge[](v21)(v22)
\Edge[](v22)(v23)
\tikzset{EdgeStyle/.style = {color=black}}
\Edge[](v6)(v9)
\Edge[](v9)(v12)
\Edge[](v12)(v15)
\Edge[](v15)(v18)
\Edge[](v18)(v21)
\Edge[](v0)(v21)
\Edge[](v0)(v3)
\Edge[](v3)(v6)
\tikzset{EdgeStyle/.style = {color=black}}
\Edge[](v7)(v10)
\Edge[](v10)(v13)
\Edge[](v13)(v16)
\Edge[](v16)(v19)
\Edge[](v19)(v22)
\Edge[](v1)(v22)
\Edge[](v1)(v4)
\Edge[](v4)(v7)
\tikzset{EdgeStyle/.style = {color=black}}
\Edge[](v8)(v11)
\Edge[](v11)(v14)
\Edge[](v14)(v17)
\Edge[](v17)(v20)
\Edge[](v20)(v23)
\Edge[](v2)(v23)
\Edge[](v2)(v5)
\Edge[](v5)(v8)
\end{tikzpicture}
\caption{ By applying edge deletions to $  \operatorname{Circ}[24,\{1,3\}] $
and relabeling the vertices, it is clear that $ C_8 \square P_3 \preceq \operatorname{Circ}[24,\{1,3\}] $.}
\label{fig:cycle_8_path_3}
\end{figure}

The next result may also be true for $ n < 24 $ but our proof needs $ n $ to be
big enough to use results from $ \operatorname{Z}( C_8 \square P_3)=
\operatorname{Z}(\operatorname{Circ}[24,\{1,3\}])=6  $.

\begin{thm}
\label{thm:circulant_multiple_of_3}
Let $ n \geq 24 $ be a multiple of $ 3 $ and let $ G =
\operatorname{Circ}[n,\{1,3\}]
$. Then $ \xi(G) = \operatorname{M}(G) = \operatorname{Z}(G) = 6$.
\end{thm}

\begin{proof}
In \hyperref[ex:cycle_8_path_3]{Example
\ref{ex:cycle_8_path_3}}, we showed that $ 6 = \xi(C_8 \square P_3) $. By
\hyperref[obs:cartesian_product_minor]{Observations \ref{obs:cartesian_product_minor}} and
\hyperref[obs:circulant_minor]{\ref{obs:circulant_minor}} $ C_8
\square P_3 \preceq C_{{n}/{3}} \square P_3 \preceq G $, and  $
\operatorname{Z}(G) \leq 6 $ by
\hyperref[prop:zero_forcing_number_non_consecutive_circulants]{Proposition
\ref{prop:zero_forcing_number_non_consecutive_circulants}}. Therefore, $ C_8
\square P_3 $ is a zero forcing minor of $ G $, and $ \xi(G) =
\operatorname{M}(G) = \operatorname{Z}(G) = 6 $ by
\hyperref[thm:m_equal_z_zeroforcing_minor_family]{Theorem
\ref{thm:m_equal_z_zeroforcing_minor_family}}. 
\end{proof}

\begin{rem}
\label{rem:M_equal_Z_mobius_ladder}
For every positive integer $ t $, $ \operatorname{Circ}[2t,\{1,t\}]  $ is the
Moebius ladder graph. The edges $ \{i,i+t\} $ of $ \operatorname{Circ}[2t,\{1,t\}]
$ are the rungs in the Moebius ladder. It was shown in \cite[Proposition
3.9]{MR2388646} that all Moebius ladder graphs have both their maximum nullity and zero forcing number equal to $ 4 $.
\end{rem}

\section{An application of vertex connectivity}
\label{sec:vertex_connectivity}
In this section, we use the known results for the vertex connectivity of a
graph to show that the maximum nullity and zero forcing number for some
circulant graphs are the same. 
\bigskip

The \textit{vertex connectivity}, denoted by $ \kappa(G) $, of a graph is the
smallest number of vertices needed to be deleted to disconnect a noncomplete
graph and $ \kappa(K_n) = n-1 $. In 2007, building on the work of
Lov\'{a}sz, Saks, Schrijver \cite{MR1770359}, \cite{MR986889},  Hein van der Holst \cite{MR2419952}
showed that the vertex connectivity of a graph is a lower bound for the maximum
nullity of a graph. Although not published, it is worth noting that in a AIM
workshop the minimum degree and vertex connectivity of a graph were used to
show that the maximum nullity is equal to the zero forcing number for certain
circulant graphs. (In
1962, Frank Harary showed that the vertex connectivity of those graphs is the
same as the minimum degree and they are now called \textit{Harary graphs}.)

\begin{thm}{\rm\cite[Theorem 4]{MR2419952}}
\label{thm:xi_bounded_by_kappa_and_max_nullity}
Let $ G $ be a graph. Then $ \kappa(G) \leq \xi(G) $.
\end{thm}

\begin{cor}
\label{cor:parameter_inequality}
Let $ G $ be a graph. Then $ \kappa(G) \leq \xi(G) \leq \operatorname{M}(G)
\leq \operatorname{Z}(G) $.
\end{cor}

\begin{obs}
\label{obs:min_degree_circulant}
Let $ G $ be a circulant graph $ \operatorname{Circ}[n,S] $ such that $ S $
does not contain $ \frac{n}{2} $. Then $ \delta(G) = 2|S|$.
\end{obs}

The circulant graph $\operatorname{Circ}[n, \{1,2,\dots,t\}]$ is called a {\em
consecutive circulant}.  It is known that a consecutive circulant is a Harary
graph (see {\rm\cite[Example 4.1.4]{West11}}), and it is shown in
{\rm\cite[Theorem 4.1.5]{West11}} that the vertex connectivity and the minimum
degree of a Harary graph are equal.

\begin{thm}{\rm\cite[Corollary 2.2]{MRcirc}}
\label{thm:M_equal_Z_consecutive_circulants}
Let $ 2t + 1 \leq n $ and let $ G = \operatorname{Circ}[n, \{1,2,\dots, t\}] $. Then 
\[
\kappa(G) = \delta(G) = \xi(G) = \operatorname{M}(G) = \operatorname{Z}(G) = 2t.
\]
\end{thm}

\begin{proof}
By \hyperref[obs:min_degree_circulant]{Observation
\ref{obs:min_degree_circulant}}, $ \delta(G) = 2t $. Since $ G $ is a Harary graph, $ \kappa(G) = \delta(G) $. By
\hyperref[cor:parameter_inequality]{Corollary \ref{cor:parameter_inequality}},
we have the following inequalities $ \kappa(G) = \delta(G) \leq \xi(G) \leq
\operatorname{M}(G) \leq \operatorname{Z}(G) $. An upper bound for the zero
forcing number of $ G $ is $ 2t $, which is given by
\hyperref[prop:zero_forcing_number_non_consecutive_circulants]{Proposition
\ref{prop:zero_forcing_number_non_consecutive_circulants}}. Therefore, $
\kappa(G) = \delta(G) = \xi(G) = \operatorname{M}(G) = \operatorname{Z}(G) =
2t. $
\end{proof}

When $ n $ is odd and $ t = \lfloor \frac{n}{2} \rfloor $ the circulant
$\operatorname{Circ}[n,\{1,2,\dots,t\}] = K_n $. The equality of $ \kappa,
\delta, \xi, \text{and}, \operatorname{Z} $ shown for consecutive circulants in
\hyperref[thm:M_equal_Z_consecutive_circulants]{Theorem
\ref{thm:M_equal_Z_consecutive_circulants}} is not true for all circulant
graphs as shown in the next example.

\begin{ex}
Let $ G $ be the graph $ \operatorname{Circ}[8,\{1,3\}] = K_{4,4}$. By
considering $ G = K_{4,4}$, we see that $
\kappa(G) = \delta(G) = 4 $ and $ \operatorname{Z}(G) = 6 $, since $
\operatorname{Z}(K_{a,b}) = a+b-2 $. It was shown in
{\rm\cite[Corollary 2.8]{MR2181887}} that $ \xi(G) = \min\{4,4\} + 1 = 5$.
\end{ex}

For $ n = 2m + 1 $, if $ n $ is prime, $ \gcd(m-1,n)= \gcd(m,n) = 1 $. So $
\operatorname{Circ}[n,\{1,\dots, m-2,m \}] \cong K_n - C_n \cong
\operatorname{Circ}[n,[m-1]].$ However $ \operatorname{Circ}[22,
\{1,2,3,4,5,6,7,8,10\}] \ncong \operatorname{Circ}[22,
\{1,2,\\3,4,5,6,7,8,9\}]$. Thus the discussion below covers graphs that are not
consecutive circulants.

\begin{prop}
\label{prop:zero_forcing_consecutive_circulant_minus_one}
Let $ H = \operatorname{Circ}[n,[m] \setminus \{m-1\}] $ where $ n > 9 $ and $ m =
\lceil n/2 \rceil - 1 $. Then $ \operatorname{Z}(H) \leq 2(m-1). $
\end{prop}

\begin{proof}
Observe first that $ 2(m-1) = \delta(H) \leq \operatorname{Z}(H) $. We will consider the case when $ n $ is odd first. Then $ n = 2m + 1   $. Since
$ m - 1 $ is not in the connection set, $ i $ is not adjacent to $ {i+(m-1)} $
or $ {i - (m-1)}$.
Note that $ i - (m-1) \equiv i + n - (m-1) \equiv i + 2m + 1 - (m
-1 ) \equiv i + m + 2 \bmod n $. It follows that $ 0 $ is not adjacent to $ {m-1} $
or $ {m+2} $, and $ 3 $ is not adjacent to $ {m+2} $ or $ {m+5} $.
Consider the set $ Z = V(H) \setminus \{{m-2}, {m-1}, {m+2} \}  $. Then $ 0 \to
{m-2} $ and $ {3} \to {m-1} $. After these two forces any vertex adjacent
to $ {m+2} $ can force $ {m+2} $, which shows that $ Z $ is a zero forcing set.
\medskip

When $ n $ is even, $ n = 2m + 2 $. Since
$ m - 1 $ and $ \frac{n}{2} $ are not in the connection set, $ i $ is not adjacent to $ {i+(m-1)} $,
$ {i - (m-1)} \equiv i+m+3$ or $ i + {n/2} \equiv i + (m + 1)$. It follows
that $ 0 $ is not adjacent to $ {m-1}$, $m+1$, or ${m+3} $, and $ 2 $ is not adjacent $ {m+1}$, ${m+3}$, or ${m+5} $. Consider the
set $ Z = V(H) \setminus \{2,{m-1},{m+3}\} $. Then $ 0 \to 2 $ and $ 2 \to
{m-1} $. Any vertex adjacent to $ {m+3} $ can force $ {m+3} $, which shows that
$ Z $ is a zero forcing set.
\end{proof}

\begin{thm}{\rm\cite[Theorem 1]{MR766498}}
\label{thm:circulant_vertex_connectivity}
Let $ G $ be a circulant graph $ \operatorname{Circ}[n,\{s_1,s_2,\dots,s_k\}]
$. There exists a proper divisor $ d $ of $ n $ such that the number of
distinct positive residues modulo $ d $ of $ s_1,s_2,
\dots,s_k,n-s_k,n-s_{k-1},\dots,n-s_1 $ is less than $
\min\{d-1,\frac{\delta(G)}{n}d\} $ if and only if $ \kappa(G) < \delta(G) $.
\end{thm}

\begin{thm}
\label{thm:circulant_delete_connection_set}
Let $ H = \operatorname{Circ}[n,[m] \setminus \{m-1\}] $ where $ n \geq 10 $ and $ m =
\lceil n/2 \rceil - 1 $. Then $\kappa(G) = \delta(G) = \xi(G) =
\operatorname{M}(G) = \operatorname{Z}(G) = 2(m - 1) $.
\end{thm}

\begin{proof}
Since $ 2(m-1) = \delta(G) = \operatorname{Z}(G) $, we need only to show $
\kappa(G) = \delta(G) $. Let $ d $ be a positive divisor of $ n $ and let $ S'
= \{1, 2, \dots, m-2,m,n-m,n-(m-2),\dots,n-1\}$. If $ d < m $, then $ d-1 \leq
m-2$ and $
1,2,\dots,d-1 $ are $ d-1 $ distinct residue of $ S' $ modulo $ d $. Note that $ d = m $ is
impossible since $ m $ does not divide $ 2m + 1 $ or $ 2m + 2 $, as $ m \geq 3
$. If $ n $ is even and  $ d = \frac{n}{2} $, then $ \frac{\delta(G)}{n}d =
\frac{\delta(G)}{2} = \frac{2(m-1)}{2} = m-1 < d-1$. Furthermore $ 1,2,\dots,
m-2,m $ are $ m-1 $ distinct residue of $ S'$ modulo $ d $ which is
greater than or equal to $ \frac{\delta(G)}{n}d $. Therefore, by
\hyperref[thm:circulant_vertex_connectivity]{Theorem
\ref{thm:circulant_vertex_connectivity}} it must be the case that $ \kappa(G)
= \delta(G)$.
\end{proof}
\section{An application of equitable partitions}
\label{sec:applications_of_equitable_partition}
In this section we use an equitable partition of a circulant graph to bound the
nullity of the graph. It fact, the lower bound is obtained from the nullity of
a circulant graph of small order which possesses the same connection set as the
circulant graph of interest.
\bigskip

An \textit{equitable partition} of a graph is a partition of the vertex set $
V_0, V_1, \dots, V_k $ such that for all $ v \in V_i $ the number $ b_{ij} $ of
neighbors in $ V_j $ is constant for all $ V_j $. Let $ V_0, V_1, \dots, V_k $
be an equitable partition of $ V(G) $. We say a \textit{divisor} of $ G $ is a
weighted directed graph with vertex set $ V_0, V_1, \dots, V_k $ and arc $
(V_i, V_j) $ having weight $ b_{ij} $ if and only if $ b_{ij} \neq 0 $.  The
matrix $ [b_{ij}] $ is the \textit{divisor matrix} associated with the
equitable partition $ V_0, V_1, \dots, V_k $. It is known that an equitable
partition of a graph $ G $ can be used to find specific eigenvalues of $ A(G)
$ (see \cite{MR2571608}). 
\bigskip

\begin{figure}[h]
\centering
\begin{tikzpicture}[scale=0.45,font=\scriptsize]
\GraphInit[vstyle=Normal]
%
\tikzset{VertexStyle/.style = {shape = circle, draw=black,
minimum size=14pt, inner sep=0.5pt}}
\Vertex[L=\hbox{$0$},x=5.55cm,y=11.1cm]{v0}
\Vertex[L=\hbox{$1$},x=4.1136cm,y=10.9109cm]{v1}
\Vertex[L=\hbox{$2$},x=2.775cm,y=10.3564cm]{v2}
\Vertex[L=\hbox{$3$},x=1.6256cm,y=9.4744cm]{v3}
\Vertex[L=\hbox{$4$},x=0.7436cm,y=8.325cm]{v4}
\Vertex[L=\hbox{$5$},x=0.1891cm,y=6.9864cm]{v5}
\Vertex[L=\hbox{$6$},x=0.0cm,y=5.55cm]{v6}
\Vertex[L=\hbox{$7$},x=0.1891cm,y=4.1136cm]{v7}
\Vertex[L=\hbox{$8$},x=0.7436cm,y=2.775cm]{v8}
\Vertex[L=\hbox{$9$},x=1.6256cm,y=1.6256cm]{v9}
\Vertex[L=\hbox{${10}$},x=2.775cm,y=0.7436cm]{v10}
\Vertex[L=\hbox{${11}$},x=4.1136cm,y=0.1891cm]{v11}
\Vertex[L=\hbox{${12}$},x=5.55cm,y=0.0cm]{v12}
\Vertex[L=\hbox{${13}$},x=6.9864cm,y=0.1891cm]{v13}
\Vertex[L=\hbox{${14}$},x=8.325cm,y=0.7436cm]{v14}
\Vertex[L=\hbox{${15}$},x=9.4744cm,y=1.6256cm]{v15}
\Vertex[L=\hbox{${16}$},x=10.3564cm,y=2.775cm]{v16}
\Vertex[L=\hbox{${17}$},x=10.9109cm,y=4.1136cm]{v17}
\Vertex[L=\hbox{${18}$},x=11.1cm,y=5.55cm]{v18}
\Vertex[L=\hbox{${19}$},x=10.9109cm,y=6.9864cm]{v19}
\Vertex[L=\hbox{${20}$},x=10.3564cm,y=8.325cm]{v20}
\Vertex[L=\hbox{${21}$},x=9.4744cm,y=9.4744cm]{v21}
\Vertex[L=\hbox{${22}$},x=8.325cm,y=10.3564cm]{v22}
\Vertex[L=\hbox{${23}$},x=6.9864cm,y=10.9109cm]{v23}
\Edge[](v0)(v1)
\Edge[](v1)(v2)
\Edge[](v2)(v3)
\Edge[](v3)(v4)
\Edge[](v4)(v5)
\Edge[](v5)(v6)
\Edge[](v6)(v7)
\Edge[](v7)(v8)
\Edge[](v8)(v9)
\Edge[](v9)(v10)
\Edge[](v10)(v11)
\Edge[](v11)(v12)
\Edge[](v12)(v13)
\Edge[](v13)(v14)
\Edge[](v14)(v15)
\Edge[](v15)(v16)
\Edge[](v16)(v17)
\Edge[](v17)(v18)
\Edge[](v18)(v19)
\Edge[](v19)(v20)
\Edge[](v20)(v21)
\Edge[](v21)(v22)
\Edge[](v22)(v23)
\Edge[](v0)(v23)
\tikzstyle{EdgeStyle}=[bend left]
\Edge[](v13)(v16)
\Edge[](v3)(v6)
\Edge[](v9)(v12)
\Edge[](v15)(v18)
\Edge[](v16)(v19)
\Edge[](v14)(v17)
\Edge[](v8)(v11)
\Edge[](v2)(v5)
\Edge[](v20)(v23)
\Edge[](v1)(v4)
\Edge[](v10)(v13)
\Edge[](v11)(v14)
\Edge[](v17)(v20)
\Edge[](v7)(v10)
\Edge[](v4)(v7)
\Edge[](v6)(v9)
\Edge[](v19)(v22)
\Edge[](v5)(v8)
\Edge[](v0)(v3)
\Edge[](v18)(v21)
\Edge[](v12)(v15)
\tikzstyle{EdgeStyle}=[bend right]
\Edge[](v0)(v21)
\Edge[](v1)(v22)
\Edge[](v2)(v23)
\end{tikzpicture}
\begin{tikzpicture}[scale=0.43,font=\scriptsize]
\GraphInit[vstyle=Normal]
%
\tikzset{VertexStyle/.style = {shape = circle, font=\large}}
\Vertex[L=\hbox{$\rightarrow$},x=-2cm,y=5cm]{v29}
\tikzset{VertexStyle/.style = {shape = circle, draw=black,
minimum size=14pt, inner sep=0.5pt}}
\Vertex[L=\hbox{${0}$},x=5.55cm,y=11.1cm]{v0}
\Vertex[L=\hbox{${1}$},x=4.1136cm,y=10.9109cm]{v1}
\Vertex[L=\hbox{${2}$},x=2.775cm,y=10.3564cm]{v2}
\Vertex[L=\hbox{${3}$},x=1.6256cm,y=9.4744cm]{v3}
\Vertex[L=\hbox{${4}$},x=0.7436cm,y=8.325cm]{v4}
\Vertex[L=\hbox{${5}$},x=0.1891cm,y=6.9864cm]{v5}
\Vertex[L=\hbox{${6}$},x=0.0cm,y=5.55cm]{v6}
\Vertex[L=\hbox{${7}$},x=0.1891cm,y=4.1136cm]{v7}
\Vertex[L=\hbox{${0'}$},x=0.7436cm,y=2.775cm]{v8}
\Vertex[L=\hbox{${1'}$},x=1.6256cm,y=1.6256cm]{v9}
\Vertex[L=\hbox{${2'}$},x=2.775cm,y=0.7436cm]{v10}
\Vertex[L=\hbox{${3'}$},x=4.1136cm,y=0.1891cm]{v11}
\Vertex[L=\hbox{${4'}$},x=5.55cm,y=0.0cm]{v12}
\Vertex[L=\hbox{${5'}$},x=6.9864cm,y=0.1891cm]{v13}
\Vertex[L=\hbox{${6'}$},x=8.325cm,y=0.7436cm]{v14}
\Vertex[L=\hbox{${7'}$},x=9.4744cm,y=1.6256cm]{v15}
\Vertex[L=\hbox{${0''}$},x=10.3564cm,y=2.775cm]{v16}
\Vertex[L=\hbox{${1''}$},x=10.9109cm,y=4.1136cm]{v17}
\Vertex[L=\hbox{${2''}$},x=11.1cm,y=5.55cm]{v18}
\Vertex[L=\hbox{${3''}$},x=10.9109cm,y=6.9864cm]{v19}
\Vertex[L=\hbox{${4''}$},x=10.3564cm,y=8.325cm]{v20}
\Vertex[L=\hbox{${5''}$},x=9.4744cm,y=9.4744cm]{v21}
\Vertex[L=\hbox{${6''}$},x=8.325cm,y=10.3564cm]{v22}
\Vertex[L=\hbox{${7''}$},x=6.9864cm,y=10.9109cm]{v23}
\Edge[](v0)(v1)
\Edge[](v1)(v2)
\Edge[](v2)(v3)
\Edge[](v3)(v4)
\Edge[](v4)(v5)
\Edge[](v5)(v6)
\Edge[](v6)(v7)
\Edge[](v7)(v8)
\Edge[](v8)(v9)
\Edge[](v9)(v10)
\Edge[](v10)(v11)
\Edge[](v11)(v12)
\Edge[](v12)(v13)
\Edge[](v13)(v14)
\Edge[](v14)(v15)
\Edge[](v15)(v16)
\Edge[](v16)(v17)
\Edge[](v17)(v18)
\Edge[](v18)(v19)
\Edge[](v19)(v20)
\Edge[](v20)(v21)
\Edge[](v21)(v22)
\Edge[](v22)(v23)
\Edge[](v0)(v23)
\tikzstyle{EdgeStyle}=[bend left]
\Edge[](v13)(v16)
\Edge[](v3)(v6)
\Edge[](v9)(v12)
\Edge[](v15)(v18)
\Edge[](v16)(v19)
\Edge[](v14)(v17)
\Edge[](v8)(v11)
\Edge[](v2)(v5)
\Edge[](v20)(v23)
\Edge[](v1)(v4)
\Edge[](v10)(v13)
\Edge[](v11)(v14)
\Edge[](v17)(v20)
\Edge[](v7)(v10)
\Edge[](v4)(v7)
\Edge[](v6)(v9)
\Edge[](v19)(v22)
\Edge[](v5)(v8)
\Edge[](v0)(v3)
\Edge[](v18)(v21)
\Edge[](v12)(v15)
\tikzstyle{EdgeStyle}=[bend right]
\Edge[](v0)(v21)
\Edge[](v1)(v22)
\Edge[](v2)(v23)
\end{tikzpicture}
\caption{ $ \operatorname{Circ}[24,\{1,3\}] $ and a relabeling showing how
the vertices can be equitably partitioned.}
\label{fig:circ_24_13}
\end{figure}
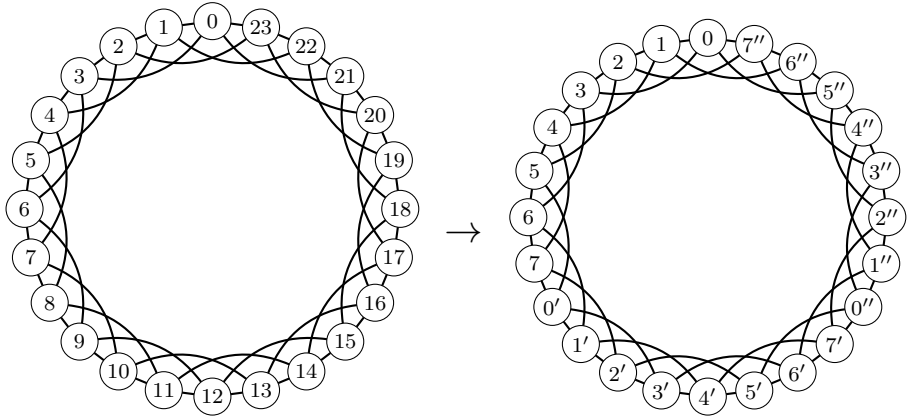

\begin{ex}
\label{ex:circ_8_13}

\hyperref[fig:circ_24_13]{Figure \ref{fig:circ_24_13}} shows the graph of the
circulant $ \operatorname{Circ}[24,\{1,3\}] $. By partitioning
the vertex set of $ \operatorname{Circ}[24,\{1,3\}] $ as in
\hyperref[fig:circ_24_13]{Figure \ref{fig:circ_24_13}}, it is clear that the
partition $ V_i = \{i, {i'}, {i''} \} $ for $ i = 0,1, \dots, 7 $ is an
equitable partition of $ \operatorname{Circ}[24,\{1,3\}] $.
\end{ex}

\begin{prop}{\rm\cite[Page 196]{MR1829620}}
\label{prop:automorphis_give_equitable_partition}
Let $ \phi $ be an automorphism of $ G $. Then the orbits of $ \phi $ give an
equitable partition of $ V(G) $. 
\end{prop}

Note that the equitable partition in \hyperref[ex:circ_8_13]{Example
\ref{ex:circ_8_13}} is obtained from the automorphism $ \varphi(i) = i + 8. $
\bigskip

\begin{thm}{\rm\cite[Theorem 3.9.5]{MR2571608}}
\label{thm:equitable_partition_eigenvalues}
Let $ G $ be a graph and let $ D $ be a divisor matrix of some equitable
partition of $ V(G) $. Then the eigenvalues of $ D $ are eigenvalues of $ A(G)
$ (including multiplicity).
\end{thm}

\begin{thm}
\label{thm:equitable_partition_non_consecutive_circulants}
Let $ G $ be the circulant graph $ \operatorname{Circ}[nk , S] $ where $ k $
is a positive integer and $ S \subseteq
\big[\big\lceil\tfrac{n}{2}\big\rceil - 1\big] $. Then
the adjacency matrix of the circulant graph $ \operatorname{Circ}[n,S]  $ is a
divisor matrix of $ G $.
\end{thm}

\begin{proof}
\label{pf:equitable_partition_non_consecutive_circulants}
The orbits of the automorphism $ \varphi(t) \equiv {t+n \bmod nk} $ of $ G $ are 
\[
V_i = \{ r \in V( \operatorname{Circ}[nk,S]) \,| \,  r \equiv i\bmod n \}.
\] 
Hence the partition $ V_0, V_1, \dots, V_{n-1} $ is an equitable partition of
$ G $.
\bigskip 

Let $ [b_{ij}] $ be the divisor matrix of $ G $ with respect to the given
equitable partition and let $ [a_{ij}]  $ be the adjacency matrix of $
\operatorname{Circ}[n,S] $. It suffices to show for all $ i $ and $j $, $
b_{ij} \leq 1 $ and $ b_{ij}$ is nonzero if and only if $ a_{ij}$ is nonzero.
Suppose $ s_1$ and $s_2$ are distinct elements in $ S $, $ V_iV_j $ is an arc, and $ {i + s_1} \in V_j $. Since  $ s_1, s_2 \in
\big[\big\lceil \tfrac{n}{2} \big\rceil - 1 \big] $, $ s_1 \pm s_2 \not\equiv 0 \bmod n$ which
implies $ i + s_1 \not\equiv i \pm s_2 \bmod n $ and $ i \pm s_2 \not\equiv j
\bmod n $. Hence $ {i \pm s_2} \notin V_j $. Also $ s_1 \in
\big[\big\lceil \tfrac{n}{2} \big\rceil - 1 \big] $, so $ 2s_1 \not\equiv 0 \bmod n $ which
implies $ i + s_1 \not\equiv i - s_1 \bmod n$ and $ i - s_1 \not\equiv j \bmod
n $.  This shows that $ {i - s_1}, {i + s_2}, {i - s_2} \notin V_j $.
Hence $ b_{ij} \leq 1 $ for all $ i $ and $j$. 
\bigskip

Suppose $V_i$ is adjacent to $ V_j $. Then there exists a vertex $ \ell \in
V_i $ and $ p \in V_j $ such that $ \ell $ is adjacent to $ p $, in $ G $. Thus, $ \ell - p \equiv i - j \bmod n $.  By definition of adjacency in $ G $,
for some $ s \in S $, $ \ell\equiv p + s \mod nk $ or $ \ell \equiv p - s \mod
nk $.  Hence $ \ell - p \equiv s \bmod nk $ or $ p - \ell \equiv s \bmod nk $.
Thus, $ i - j \equiv \ell - p \equiv s \bmod n $ or $ i - j \equiv \ell - p
\equiv -s \bmod n $. In either case, $ i $ is adjacent to $ j $ in $
\operatorname{Circ}[n,S] $. Now suppose $ i $ is adjacent to $ j $ in
$\operatorname{Circ}[n,S] $ where $ 0 \leq i,j \leq n-1 $ as integers. Then it
must be the case that $ j = i + s \bmod n $ or $ j = i - s \bmod n $.  In $
\operatorname{Circ}[nk,S] $, $ i \in V_i $ and $ {i+s} \in V_j $ or $ {i-s} \in V_j $. In either case, $ b_{ij} \neq 0 $ in the divisor matrix of $G$. 
\end{proof}

When $ n $ is even in
\hyperref[thm:equitable_partition_non_consecutive_circulants]{Theorem
\ref{thm:equitable_partition_non_consecutive_circulants}} the connection set
cannot be extended to include $ \frac{n}{2} $.
\bigskip

\begin{ex}
\label{ex:even_case_divisor_matrix}
Let $ G = \operatorname{Circ}[12,\{1,3\}] $ and $ H =
\operatorname{Circ}[6,\{1,3\}] = K_{3,3} $. Furthermore, using the equitable
partition described in the proof of
\hyperref[thm:equitable_partition_non_consecutive_circulants]{Theorem
\ref{thm:equitable_partition_non_consecutive_circulants}}, 
\[
V_0 = \{0,6\} , V_1 = \{1,7\}, V_2 = \{2,8\}, V_3 = \{3,9\}, 
\]
$ b_{0,1} = 1, b_{0,2} = 0, b_{0,3} = 2 $, 
\[
[b_{ij}]=
\begin{pmatrix}
0 & 1 & 0 & 2 & 0 & 1 \\
1 & 0 & 1 & 0 & 2 & 0 \\
0 & 1 & 0 & 1 & 0 & 2 \\
2 & 0 & 1 & 0 & 1 & 0 \\
0 & 2 & 0 & 1 & 0 & 1 \\
1 & 0 & 2 & 0 & 1 & 0 \\
\end{pmatrix} ,
\]
and $ A(\operatorname{Circ}[6,\{1,3\}]) $ is not the divisor matrix of $ G $. By
computation, the eigenvalues of $ A(G) $ are $ \pm 4, \pm \sqrt{3}, \pm 1, 0 $ and
$ H $ is bipartite and $ 3 $ - regular which implies $ \pm 3 $ are eigenvalues
of $ H $. This shows that the adjacency matrix of $ H $ is not a divisor matrix
of $ G $.
\end{ex}

The next corollary is a direct result of
\hyperref[thm:equitable_partition_eigenvalues]{Theorem
\ref{thm:equitable_partition_eigenvalues}} and
\hyperref[thm:equitable_partition_non_consecutive_circulants]{Theorem
\ref{thm:equitable_partition_non_consecutive_circulants}}.

\begin{cor}
\label{cor:spectrum_subset_nonconsective_circulants}
Consider the circulant graph $ \operatorname{Circ}[nk , S] $ where $ k $ is a
positive integer and $ S \subseteq \big[\big\lceil\tfrac{n}{2}\big\rceil -1\big]
$. Then 
\[
\operatorname{spec}(A(\operatorname{Circ}[n,S])) \subseteq
\operatorname{spec}(A(\operatorname{Circ}[nk,S]))
\]
and $ \operatorname{null}(A(\operatorname{Circ}[n,S])) \leq
\operatorname{null}(A(\operatorname{Circ}[nk,S])) $.
\end{cor}

It was shown in \hyperref[thm:circulant_multiple_of_3]{Theorem
\ref{thm:circulant_multiple_of_3}} that $ \operatorname{M}(
\operatorname{Circ}[3k,\{1,3\}]) = \operatorname{Z}(
\operatorname{Circ}[3k,\{1,3\}]) = 6 $ for $ k \geq 8 $. The next result
establishes field independence, in addition to showing that the maximum nullity
equals the zero forcing number for many additional circulants..
\bigskip

\begin{thm}
\label{thm:m_equals_z_circ_8_13}
Let $ k $ be a positive integer and let $ \ell $ be an odd integer between $ 3
$ and $ 21 $. Then
\[
\operatorname{M}(\operatorname{Circ}[(\ell^2-1)k,\{1,\ell\}])=
\operatorname{Z}(\operatorname{Circ}[(\ell^2-1)k,\{1,\ell\}])=2\ell,
\]
$\operatorname{Circ}[(\ell^2-1)k,\{1,\ell\}]$ has field independent minimum
rank, and its adjacency matrix is an universally optimal matrix.
\end{thm}

\begin{proof}
\label{pf:m_equals_z_circ_8_13}
Let $ n = \ell^2-1 $, $ S = \{1,\ell\} $, and $ G = \operatorname{Circ}[nk,S] $
for $ k \geq 1. $ By
\hyperref[prop:zero_forcing_number_non_consecutive_circulants]{Proposition
\ref{prop:zero_forcing_number_non_consecutive_circulants}} and
\hyperref[prop:M_is_less_than_Z_over_any_field]{Proposition
\ref{prop:M_is_less_than_Z_over_any_field}}, $ \operatorname{M}(G) \leq \operatorname{Z}(G) \leq 2\ell $. Thus it
suffices to show that $ \operatorname{null}(A( \operatorname{Circ}[n,S])) =
2\ell. $ This is easily verified using computer software. (SageMath offers
commands for computing the adjacency matrix of a graph and its nullity.)
\end{proof}

\begin{conj}
\label{conj:m_equals_z_circ}
For all positive values of $ k $ and odd $ \ell $,
\[
\operatorname{M}(\operatorname{Circ}[(\ell^2-1)k,\{1,\ell\}])=
\operatorname{Z}(\operatorname{Circ}[(\ell^2-1)k,\{1,\ell\}])=2\ell.
\]
and field independent minimum rank with universally optimal matrix $ A(G) $.
\end{conj}
\section{An application of equitable decompositions}
\label{sec:applications_of_equitable_decompositions}
In this section, we use the equitable decomposition, introduced in
\cite{MR3573808}, of the adjacency matrix to
establish field independent minimum rank of a graph. The graphs of interest are
the extended cube graphs $ \operatorname{ECG}(6q+1,6q+1) $ where $ q $ is a
nonnegative integer.
\bigskip

An \textit{automorphism} of a graph $ G $ is an isomorphism $ \phi $ from $
V(G) $ to $ V(G) $ such that $ \phi(i) $ is adjacent to $ \phi(j) $ if and only
if $ i $ is adjacent to $ j $. Let $ G  $ be a graph with $ v,u \in V(G) $ and
let $ \phi $ be an automorphism of $ G $. Define the relation $ \approx $ on
the vertices of $ G $ by $ v \approx u$ if and only if there exists a
nonnegative integer $ j $ for which $ v = \phi^j(u) $. This relation is an
equivalence relation on the vertices of $ G $ and the equivalence classes are
the \textit{orbits} of $ \phi $. Let $ \phi $ be an uniform automorphism of $ G
$ with orbit size $ k $ where $ 1 < k $. A \textit{transversal} of $ \phi $ is
a subset of $ V(G) $ containing exactly one vertex from each orbit of $ \phi $.
The $ \ell $-power of transversal $ T $ is defined to be the following
transversal, 
\[
T_\ell = \{ \phi^\ell(v)\, |\, v \in T \} 
\]
for $ \ell \in \{ 0,1,2,\dots,k-1\} $. It is straightforward to see that $
T_\ell $ is a transversal and $ \cup_{\ell = 0}^{k-1} T_\ell = V(G) $.
\bigskip

Given an automorphism $ \phi $, an $ n \times n $ matrix $ A = [a_{ij}] $
associated with the graph $ G $ on $ n $ vertices such that  \[
a_{\phi(i),\phi(j)} = a_{ij} \] for all $ i,j \in \{ 1,2,\dots,n\} $, is called
$ \phi-$\textit{compatible}. An $ n \times n $ matrix $ A $ associated with the
graph $ G $ is called $ \phi-$\textit{automorphism compatible} if it is $
\phi-$compatible for every automorphism $ \phi $ of $ G $. Recently in 2017,
Barrett et al. used equitable partitions of a graph in \cite{MR3573808} to
decompose $ A(G) $.  This decomposition can be used to determine all
eigenvalues of $ A(G) $.  As a result, this decomposition is useful for
determining a lower bound for the maximum nullity.  Moreover, it can be use to
establish a potential candidate for an universally optimal matrix.

\begin{ex}
\label{ex:extende_cube_not_field_indenpendent}
In general, the extended cube graphs do not have field independent minimum rank. Some extended cube graphs are isomorphic to the Cartesian
product of a cycle and a path. For instance, $ \operatorname{ECG}(0,3) $ is
isomorphic to $ C_7 \square P_2 $. It was shown in
\hyperref[ex:generalized_petersen_not_field_independent]{Example
\ref{ex:generalized_petersen_not_field_independent}} that
$\operatorname{mr}(\mathbb{Z}_2, C_7 \square P_2) \neq \operatorname{mr}(C_7
\square P_2) $.
\end{ex}

\begin{obs}
The adjacency matrix of a graph is automorphism compatible.
\end{obs}

The next theorem is stated in \cite{MR3573808} for automorphism compatible
matrices, but as noted there it could be stated for a $ \phi-$compatible matrix
and we do so.
\bigskip

\begin{thm}{\rm\cite[Theorem 3.8]{MR3573808}}
\label{thm:equitable_decomposition_of_graphs}
Let $ G $ be a graph on $ n $ vertices, let $ \phi $ be an uniform
automorphism of $ G $ of orbit size $ k $, let $ \operatorname{T_0} $ be a transversal
of the orbits of $ \phi $, and let $ A $ be an $ \phi-$compatible matrix
in $ \mathcal{S}(G) $. Set $ A_\ell = A[ \operatorname{T}_0, \operatorname{T}_\ell] $, $ \ell = 0,
1, \dots, k-1$, let $ \omega = e^{2\pi i/k} $, and define 
\[
B_j = \sum_{\ell=0}^{k-1} \omega^{j\ell}A_\ell, \quad j=0,1,\dots,k-1. 
\]
Then for some invertible matrix $ S $
\begin{equation}
\label{eq:eq_decom}
S^{-1}AS = B_0 \oplus B_1 \oplus \cdots \oplus B_{k-1}
\end{equation}
and 
\[
\sigma(A) = \sigma(B_0) \cup \sigma(B_1) \cup \cdots \cup \sigma(B_{k-1}).
\] The decomposition in \hyperref[eq:eq_decom]{(\ref{eq:eq_decom})}
is called an equitable decomposition of $ A $. 
\end{thm}

\begin{obs}
\label{obs:bidiakis_cube_automorphism}
Let $ G $ be a extended cube graph $ \operatorname{ECG}(t,t) $ on $ n $ vertices and let $ r =
\frac{n}{4}$. Then the function $ \varphi(x) \equiv x + r \mod n $
is a uniform automorphism for $ G $. The function $ \varphi $ can also be written as a
permutation,
\[
\scalebox{0.8}{
$
\phi = (0,0+r,0+2r,0+3r)(1,1+r,1+2r,1+3r) \cdots (r-1,r-1+r,r-1+2r,r-1+3r).
$
}
\]
Furthermore, $ T_0 = \{0,1,\dots,r-1\} $ is a transversal.
\end{obs}

\begin{ex}
\label{ex:bidiakis_cube}
The following is an example of constructing the eigenvalues of $ \operatorname{ECG}(1,1) $ using an equitable decomposition. As in
\hyperref[obs:bidiakis_cube_automorphism]{Observation
\ref{obs:bidiakis_cube_automorphism}}, 

\[
\varphi(x) \equiv x + 3 \mod 12, 
\]
is an automorphism with permutation representation $
\phi=(0,3,6,9)(1,4,7,10)(2,5,8,11) $, and the transversals are $ T_0 =
\{0,1,2\} $, $ T_1 = \{3,4,5\} $, $ T_2 = \{6,7,8\} $, $ T_3 = \{9,10,11\} $.
Let

\[
\scalebox{0.8}{
$
A_0 = 
\begin{blockarray}{cccc}
 & \phi^0(0)=0 & \phi^0(1)=1 & \phi^0(2)=2 \\
\begin{block}{c(ccc)}
0 & 0 & 1 & 0  \\
1 & 1 & 0 & 1  \\
2 & 0 & 1 & 0  \\
\end{block}
\end{blockarray} \quad ,
\quad
A_1 =
\begin{blockarray}{cccc}
 & \phi^1(0)=3 & \phi^1(1)=4 & \phi^1(2)=5 \\
\begin{block}{c(ccc)}
0 & 0 & 0 & 0  \\
1 & 0 & 0 & 0  \\
2 & 1 & 0 & 0  \\
\end{block}
\end{blockarray}\quad ,
$
}
 \]
\[
\scalebox{0.8}{
$
A_2 = 
\begin{blockarray}{cccc}
 & \phi^2(0)=6 & \phi^2(1)=7 & \phi^2(2)=8 \\
\begin{block}{c(ccc)}
0 & 0 & 0 & 1  \\
1 & 0 & 1 & 0  \\
2 & 1 & 0 & 0  \\
\end{block}
\end{blockarray}\quad ,
\quad
\text{ and }
\quad
A_3 = 
\begin{blockarray}{cccc}
 & \phi^3(0)=9 & \phi^3(1)=10 & \phi^3(2)=11 \\
\begin{block}{c(ccc)}
0 & 0 & 0 & 1  \\
1 & 0 & 0 & 0  \\
2 & 0 & 0 & 0  \\
\end{block}
\end{blockarray} \,.
$
}
 \]

Hence

\[
B_0 = A_0 + A_1 + A_2 + A_3 =
\begin{pmatrix}
0 & 1 & 2 \\
1 & 1 & 1 \\
2 & 1 & 0 \\
\end{pmatrix}
\]
and it follows that the spectrum of $ B_0 $ is $ \{ 3,0,-2 \}$ with eigenvector $ x_0 = [1,-2,1]^T $
corresponding to the eigenvalue $ 0 $. Also,
\[
B_1 = A_0 + iA_1 - A_2 - iA_3 =
\begin{pmatrix}
0 & 1 & -1-i \\
1 & -1 & 1 \\
-1+i & 1 & 0 \\
\end{pmatrix}
\]
and the spectrum of $ B_1 $ is approximately $\{ 1.561552,0,-2.561552 \}$ with eigenvector $ x_1 =
[i,1+i,1]^T $
corresponding to the eigenvalue $ 0 $,
\[
B_2 = A_0 - A_1 + A_2 - A_3 =
\begin{pmatrix}
0 & 1 & 0 \\
1 & 1 & 1 \\
0 & 1 & 0 \\
\end{pmatrix}
\]
has spectrum $ \{2,0,-1\} $ with eigenvector $ x_2 = [1,0,-1]^T $
corresponding to the eigenvalue $ 0 $, and
\[
B_3 = A_0 - iA_1 - A_2 + iA_3 =
\begin{pmatrix}
0 & 1 & -1+i \\
1 & 1 & 1 \\
-1-i & 1 & 0 \\
\end{pmatrix}
\]
has spectrum approximately $ \{1.561552,0,-2.561552\} $ with
eigenvector
$ x_3 =
[-1,-1-i,i]^T $
corresponding to the eigenvalue $ 0 $. Using SageMath (see \cite{ecgsage}), we compute
the eigenvalues of $ \operatorname{ECG}(1,1) $ to be approximately 
\[
\{
3,2,1.561552,1.561552,0,0,0,0,-1,-2,-2.561552,-2.561552\}
\]
which is the union of the spectra of $ B_0,B_1,B_2,B_3 $.

\end{ex}

\begin{thm}
\label{thm:generalized_bidiakis_cube_m_z}
Let $ G $ be a extended cube graph $\operatorname{ECG}(6q+1,6q+1) $ for some
nonnegative integer $ q $.  Then $ G $ has field independent minimum rank and $ A(G) $ is a universally optimal matrix.
\end{thm}

\begin{proof}
\label{proof:generalized_bidiakis_cube_m_z}
First we will show that the adjacency matrix of each such extended cube graph
has nullity at least 4. Hence by
\hyperref[cor:extended_bidiakis_max_nullity_zero_forcing]{Corollary
\ref{cor:extended_bidiakis_max_nullity_zero_forcing}} the adjacency matrix
realizes the maximum nullity.
\bigskip

It was shown in  \hyperref[ex:bidiakis_cube]{Example \ref{ex:bidiakis_cube}}
that the nullity of $ \operatorname{ECG}(1,1) $ has nullity
equal to 4, so we assume $ q > 0 $. Let $ G $ be a extended cube graph $
\operatorname{ECG}(6q+1,6q+1) $ and let $ n  $ be the number of vertices
of $ G $. Consider the uniform automorphism

\[
\varphi(x) = x + r \mod n
\]
where $ r = \frac{n}{4} $ given by \hyperref[obs:bidiakis_cube_automorphism]{Observation
\ref{obs:bidiakis_cube_automorphism}}. By \hyperref[thm:equitable_decomposition_of_graphs]{Theorem
\ref{thm:equitable_decomposition_of_graphs}}, $ G $ has the
following spectrum 
\[
\operatorname{spec}(A(G)) = \operatorname{spec}(B_0) \cup \operatorname{spec}(B_1) \cup
\operatorname{spec}(B_2) \cup \operatorname{spec}(B_3)
\]
for the matrices $ B_i $ corresponding to $ \varphi $. We show that $
B_0,B_1,B_2,B_3 $ each have nullity at least 1, which implies $ A(G) $ has
nullity at least 4.
\bigskip

\newcommand{\0}{0}
\newcommand{\A}{\tilde{A}}
\newcommand{\xb}{\tilde{x}}
\newcommand{\xh}{\hat{x}}
\newcommand{\bdots}{\reflectbox{$\ddots$}}

The transversals with respect to $ \varphi $ are $ T_0 = \{0,1,2,\dots,r-1\} $,
$ T_1 = \{r,r+1,r+2,\dots,2r-1\} $, $ T_2 = \{2r,2r+1,2r+2,\dots,3r-1\} $, $
T_3 = \{3r,3r+1,3r+2,\dots,4r-1\} $. Hence $ k=4 $ and $ \omega = e^{2\pi i/4}
= i $.  For the graph $ \operatorname{ECG}(1,1) $, let $ \A_0,\A_1,\A_2,\A_3 $ be the
corresponding  matrices used in
\hyperref[thm:equitable_decomposition_of_graphs]{Theorem
\ref{thm:equitable_decomposition_of_graphs}} to construct $ \tilde{B}_0,
\tilde{B}_1, \tilde{B}_2, \tilde{B}_3 $ such that 
\[
\operatorname{spec}(A( \operatorname{ECG}(1,1))) =
\operatorname{spec}(\tilde{B}_0) \cup \operatorname{spec}(\tilde{B}_1) \cup
\operatorname{spec}(\tilde{B}_2) \cup \operatorname{spec}(\tilde{B}_3)
\]
and
$ \tilde{B}_0 = \A_0 + \A_1 + \A_2 + \A_3 $. Also, let $ \xb_0, \xb_1, \xb_2 $
be the eigenvectors of $ \tilde{B}_0,\tilde{B}_1, \tilde{B}_2 $ respectively,
corresponding to the eigenvalue 0. It follows that $ A_0,A_1,A_2,A_3 $ are the
matrices
\bigskip
\begin{center}
\scalebox{0.9}{
\begin{tikzpicture}
\node at (-6.5,0) {$A_0 = 
\begin{pmatrix}
\tilde{A}_0 & \tilde{A}_1  & 0 & 0 & 0 & \cdots &
0 \\
\tilde{A}_3  & \tilde{A}_0 & \tilde{A}_1 & 0 & 0 & \cdots &
0 \\
0 & \tilde{A}_3 & \tilde{A}_0 & \tilde{A}_1 & 0 & \cdots &
0 \\
\vdots & \vdots & \ddots & \ddots & \ddots & \vdots & \vdots \\
0 & \cdots & 0 & \tilde{A}_3 & \tilde{A}_0 & \tilde{A}_1 &
0  \\
0 & \cdots & 0 & 0 & \tilde{A}_3 & \tilde{A}_0 & \tilde{A}_1  \\
0 & \cdots & 0 & 0 & 0 & \tilde{A}_3 &
\tilde{A}_0   \\
\end{pmatrix}
,$};
\node at (0,0) {$ $};
\end{tikzpicture}
}
\end{center}
\bigskip
\begin{center}
\scalebox{0.9}{
\begin{tikzpicture}
\node at (-7,0) {$ A_1 = 
\begin{pmatrix}
0 & 0  & 0 & 0 & 0 & \cdots &
0 \\
0  & 0 & 0 & 0 & 0 & \cdots &
0 \\
0 & 0 & 0 & 0 & 0 & \cdots &
0 \\
\vdots & \vdots & \vdots & \ddots & \vdots & \vdots & \vdots \\
0 & 0 & 0 & 0 & 0 & \cdots &
0  \\
0 & 0 & 0 & 0 & 0 & \cdots &
0  \\
\tilde{A}_1 & 0 & 0 & 0 & 0 & \cdots & 0   \\
\end{pmatrix},
$};
\node at (0,0) {$ $};
\end{tikzpicture}
}
\end{center}
\bigskip

\begin{center}
\scalebox{0.9}{
\begin{tikzpicture}
\node at (0,0) {$A_2 = 
\begin{pmatrix}
0 & 0  & 0 & \cdots & 0 & 0 &
\tilde{A}_2 \\
0  & 0 & 0 & \cdots & 0 & \tilde{A}_2 &
0 \\
0 & 0 & 0 & \cdots & \tilde{A}_2 & 0 &
0 \\
\vdots & \vdots & \vdots & \reflectbox{$\ddots$} & \vdots & \vdots & \vdots \\
0 & 0 & \tilde{A}_2 & \cdots & 0 & 0 &
0  \\
0 & \tilde{A}_2 & 0 & \cdots & 0 & 0 &
0  \\
\tilde{A}_2 & 0 & 0 & \cdots & 0 & 0 & 0   \\
\end{pmatrix},
$};
\node at (4.3,0) {and};
\node at (5.7,0) {$A_3 = A_1^T$ .};
\end{tikzpicture}
}
\end{center}

By definition,
\begin{equation}
\label{eq: bj_matrix}
B_j = i^{0j}A_0 + i^jA_1 + i^{2j}A_2 + i^{3j}A_3
= A_0 + i^jA_1 + (-1)^{j}A_2 + i^{3j}A_3
\end{equation} 
for $j = 0, 1, 2, 3 $, so the following matrices are constructed
\begin{align*}
 B_0 &= A_0 + A_1 + A_2 + A_3 
 &&B_1 = A_0 + iA_1 - A_2 - iA_3 \\
 B_2 &= A_0 - A_1 + A_2 - A_3 
 &&B_3 = A_0 - iA_1 - A_2 + iA_3.\\
\end{align*}
Writing $ B_j $ in terms of the matrices $ \A_0, \A_1, \A_2, \A_3 $ we get the following matrix

\[
\scalebox{1}{
$
\left(
\begin{smallmatrix}
\A_0 & \A_1  & \0 & \0 & \0 & \0 & \cdots & \0 & \0 & \0 & (-1)^j\A_2 + i^{3j}\A_3 \\
\A_3 & \A_0  & \A_1 & \0 & \0 & \0 & \cdots & \0 & \0 & (-1)^j\A_2 & \0 \\
\0 & \A_3  & \A_0 & \A_1 & \0 & \0 & \cdots & \0 & (-1)^j\A_2 & \0 & \0 \\
\vdots &  & \ddots & \ddots & \ddots &  &  & \bdots & & &\vdots\\
\0 & \cdots  & \0 & \A_3 & \A_0 & \A_1 & (-1)^j\A_2 & \0 & \0 & \cdots & \0 \\
\0 & \cdots  & \0 & \0 & \A_3 & \A_0 + (-1)^j\A_2 & i^{0j}\A_1  & \0 & \0 & \cdots & \0 \\
\0 & \cdots  & \0 & \0 & (-1)^j\A_2 & \A_3 & \A_0 & \A_1 & \0 & \cdots & \0 \\
\vdots &  &  & \bdots &  &  & \ddots & \ddots & \ddots & & \vdots\\
\0 & \0  & (-1)^j\A_2 & \0 & \cdots & \0 & \0 & \A_3 & \A_0 & \A_1 & \0 \\
\0 & (-1)^j\A_2  & \0 & \0 & \cdots & \0 & \0 & \0 & \A_3 & \A_0 & \A_1 \\
i^j\A_1+(-1)^j\A_2 & \0  & \0 & \0 & \cdots & \0 & \0 & \0 & \0 & \A_3 & \A_0 \\
\end{smallmatrix}
\right).
$
}
\]

For simplicity of notation let $ \xh_1 = -\A_2\xb_1 $ and $ \xh_2 = -\xb_2 $. We show that 
$ x_0 = \bigoplus_{m=1}^{2q+1} \tilde{x}_0 $, 
$ x_1 = \bigoplus_{m=1}^{q}  (\tilde{x}_1 \oplus \hat{x}_1) \oplus \tilde{x}_1
$, and
$ x_2 = \bigoplus_{m=1}^{q}  (\tilde{x}_2 \oplus \hat{x}_2) \oplus \tilde{x}_2 $, 
are eigenvectors corresponding to eigenvalue 0 for $ B_0 $, $ B_1 $, and $ B_2 $
respectively. Since $ B_3 = B_1^T $ it follows that $ B_3 $ also has a
zero eigenvalue so we omit showing that $ B_3 $ has an eigenvalue of zero.
Observe that

\[
B_0 =
\left(
\begin{smallmatrix}
\A_0 & \A_1  & \0 & \0 & \0 & \0 & \cdots & \0 & \0 & \0 & \A_2 + \A_3 \\
\A_3 & \A_0  & \A_1 & \0 & \0 & \0 & \cdots & \0 & \0 & \A_2 & \0 \\
\0 & \A_3  & \A_0 & \A_1 & \0 & \0 & \cdots & \0 & \A_2 & \0 & \0 \\
\vdots &  & \ddots & \ddots & \ddots &  &  & \bdots & & &\vdots\\
\0 & \cdots  & \0 & \A_3 & \A_0 & \A_1 & \A_2 & \0 & \0 & \cdots & \0 \\
\0 & \cdots  & \0 & \0 & \A_3 & \A_0 + \A_2 & \A_1  & \0 & \0 & \cdots & \0 \\
\0 & \cdots  & \0 & \0 & \A_2 & \A_3 & \A_0 & \A_1 & \0 & \cdots & \0 \\
\vdots &  &  & \bdots &  &  & \ddots & \ddots & \ddots & & \vdots\\
\0 & \0  & \A_2 & \0 & \cdots & \0 & \0 & \A_3 & \A_0 & \A_1 & \0 \\
\0 & \A_2  & \0 & \0 & \cdots & \0 & \0 & \0 & \A_3 & \A_0 & \A_1 \\
\A_1+\A_2 & \0  & \0 & \0 & \cdots & \0 & \0 & \0 & \0 & \A_3 & \A_0 \\
\end{smallmatrix}
\right).
\]
The product $ B_0x_0 $ reduces down to the following vector
\[
\begin{pmatrix}
\A_0\xb_0 + \A_1\xb_0  + \A_2\xb_0 + \A_3\xb_0 \\
\A_0\xb_0 + \A_1\xb_0  + \A_2\xb_0 + \A_3\xb_0 \\
\vdots\\
\A_0\xb_0 + \A_1\xb_0  + \A_2\xb_0 + \A_3\xb_0 \\
\A_0\xb_0 + \A_1\xb_0  + \A_2\xb_0 + \A_3\xb_0 \\
\end{pmatrix}
=
\begin{pmatrix}
\tilde{B}_0\xb_0\\
\tilde{B}_0\xb_0\\
\vdots\\
\tilde{B}_0\xb_0\\
\tilde{B}_0\xb_0\\
\end{pmatrix}
=
0,
\]
since $ \tilde{B}_0 = \tilde{A}_0 + \tilde{A}_1 + \tilde{A}_2 + \tilde{A}_3 $. 
\bigskip

To compute $ B_1x_1 $ consider the fact that $ \tilde{x}_1 = [i,1+i,1]^T $ is an
eigenvector for $ \tilde{B}_1 $. So by definition, $ \hat{x}_1 = [-1,-1-i,-i]^T $
and 
\[
\tilde{A}_0\hat{x}_1 =
\begin{pmatrix}
0 & 1 & 0 \\
1 & 0 & 1 \\
0 & 1 & 0 
\end{pmatrix}
\begin{pmatrix}
-1  \\
-1-i  \\
-i  
\end{pmatrix}
=-A_0\tilde{x}_1
\]
\[
\tilde{A}_1\hat{x}_1 =
\begin{pmatrix}
0 & 0 & 0 \\
0 & 0 & 0 \\
1 & 0 & 0 
\end{pmatrix}
\begin{pmatrix}
-1  \\
-1-i  \\
-i  
\end{pmatrix}
=
\begin{pmatrix}
0  \\
0  \\
-1  
\end{pmatrix}
=
i\begin{pmatrix}
0 & 0 & 0 \\
0 & 0 & 0 \\
1 & 0 & 0 
\end{pmatrix}
\begin{pmatrix}
i  \\
1+i  \\
1  
\end{pmatrix}
=i\tilde{A}_1\tilde{x}_1.
\]
Since $ \tilde{A}_2^2 = I $ it follows that $ \tilde{A}_2\hat{x}_1 = -\tilde{x}_1 $.
Also,
\[
\tilde{A}_3\hat{x}_1 =
\begin{pmatrix}
0 & 0 & 1 \\
0 & 0 & 0 \\
0 & 0 & 0 
\end{pmatrix}
\begin{pmatrix}
-1  \\
-1-i  \\
-i  
\end{pmatrix}
=
\begin{pmatrix}
-i  \\
0  \\
0  
\end{pmatrix}
=
-i\begin{pmatrix}
0 & 0 & 1 \\
0 & 0 & 0 \\
0 & 0 & 0 
\end{pmatrix}
\begin{pmatrix}
i  \\
1+i  \\
1  
\end{pmatrix}
=-i\tilde{A}_3\tilde{x}_1.
\]

In other words, 
\[
\tilde{A}_0\xh_1 = (-1-i)\mathbf{1}^T \, ,  \A_1\xh_1 =
i\A_1\xb_1 \, ,  \A_2\xh_1 = -\xb_1 \, , \text{ and }  \A_3\xh_1 = -i\A_3\xb_1
\]
and these values are used to reduce the entries of the next product. We have
that $ B_1x_1 $ is
\[
\left(
\begin{smallmatrix}
\A_0 & \A_1  & \0 & \0 & \0 & \0 & \cdots & \0 & \0 & \0 & -\A_2 - i\A_3 \\
\A_3 & \A_0  & \A_1 & \0 & \0 & \0 & \cdots & \0 & \0 & -\A_2 & \0 \\
\0 & \A_3  & \A_0 & \A_1 & \0 & \0 & \cdots & \0 & -\A_2 & \0 & \0 \\
\vdots &  & \ddots & \ddots & \ddots &  &  & \bdots & & &\vdots\\
\0 & \cdots  & \0 & \A_3 & \A_0 & \A_1 & -\A_2 & \0 & \0 & \cdots & \0 \\
\0 & \cdots  & \0 & \0 & \A_3 & \A_0 + -\A_2 & \A_1  & \0 & \0 & \cdots & \0 \\
\0 & \cdots  & \0 & \0 & -\A_2 & \A_3 & \A_0 & \A_1 & \0 & \cdots & \0 \\
\vdots &  &  & \bdots &  &  & \ddots & \ddots & \ddots & & \vdots\\
\0 & \0  & -\A_2 & \0 & \cdots & \0 & \0 & \A_3 & \A_0 & \A_1 & \0 \\
\0 & -\A_2  & \0 & \0 & \cdots & \0 & \0 & \0 & \A_3 & \A_0 & \A_1 \\
i\A_1+-\A_2 & \0  & \0 & \0 & \cdots & \0 & \0 & \0 & \0 & \A_3 & \A_0 \\
\end{smallmatrix}
\right)
x_1.
\]
We show that the product $ B_1x_1 $ is given by the vector,
\[
\begin{pmatrix}
\A_0\xb_1 + \A_1\xh_1  + (-\A_1 - i\A_3)\xb_1 \\
--------------- \\
\A_3\xb_1 + \A_0\xh_1  + \A_1\xb_1 + (-\A_2\xh_1) \\
\A_3\xh_1 + \A_0\xb_1  + \A_1\xh_1 + (-\A_2\xb_1) \\
\A_3\xb_1 + \A_0\xh_1  + \A_1\xb_1 + (-\A_2\xh_1) \\
\A_3\xh_1 + \A_0\xb_1  + \A_1\xh_1 + (-\A_2\xb_1) \\
\vdots\\
\A_3\xh_1 + \A_0\xb_1  + \A_1\xh_1 + (-\A_2\xb_1) \\
\A_3\xb_1 + \A_0\xh_1  + \A_1\xb_1 + (-\A_2\xh_1) \\
--------------- \\
\A_3\xh_1 + (\A_0 - \A_2)\xb_1 + \A_1\xh_1 \\
--------------- \\
\A_3\xh_1 + \A_0\xb_1  + \A_1\xh_1 + (-\A_2\xb_1) \\
\A_3\xb_1 + \A_0\xh_1  + \A_1\xb_1 + (-\A_2\xh_1) \\
\vdots\\
\A_3\xb_1 + \A_0\xh_1  + \A_1\xb_1 + (-\A_2\xh_1) \\
\A_3\xh_1 + \A_0\xb_1  + \A_1\xh_1 + (-\A_2\xb_1) \\
--------------- \\
(i\A_1 - \A_2)\xb_1 + \A_3\xh_1  + \A_0\xb_1
\end{pmatrix}
= 0
\]
which implies that both matrices $ B_1 $ and $ B_3 $ has a zero eigenvalue.
Note the that each entry in the product takes on one of the following
values,
\begin{align*}
\A_0\xb_1+\A_1\xh_1+(-\A_2-i\A_3)\xb_1&=\A_0\xb_1+i\A_1\xb_1-\A_2\xb_1-i\A_3\xb_1
\nonumber\\
&=(\A_0+i\A_1-\A_2-i\A_3)\xb_1 =\tilde{B}_1\tilde{x}_1 = 0 \nonumber\\
\A_3\xb_1 + \A_0\xh_1  + \A_1\xb_1 +
(-\A_2\xh_1)&= (1,0,0)^T + (-1-i,-1-i,-1-i)^T \nonumber \\
&\quad +(0,0,i)^T +(i,1+i,1)^T  = 0 \nonumber\\
\A_3\xh_1 + \A_0\xb_1  + \A_1\xh_1 +
(-\A_2\xb_1)&=-i\A_3\xb_1+\A_0\xb_1+i\A_1\xb_1-\A_2\xb_1\\
&=(\A_0+i\A_1-\A_2-i\A_3)\xb_1 =\tilde{B}_1\tilde{x}_1 = 0 \nonumber\\
\A_3\xh_1 + (\A_0 - \A_2)\xb_1 + \A_1\xh_1 &= 0 \quad \text{ by }
\hyperref[eq: bj_matrix]{(\ref{eq: bj_matrix})}  \nonumber\\
(i\A_1 - \A_2)\xb_1 + \A_3\xh_1  + \A_0\xb_1 &= i\A_1\xb_1 - \A_2\xb_1 -
i\A_3\xb_1 + \A_0\xb_1 \nonumber \\
&=(\A_0+i\A_1-\A_2-i\A_3)\xb_1 =\tilde{B}_1\tilde{x}_1 = 0. \nonumber
\end{align*}
Finally, we compute $ B_2x_2 $,

\[
\left(
\begin{smallmatrix}
\A_0 & \A_1  & \0 & \0 & \0 & \0 & \cdots & \0 & \0 & \0 & \A_2 - \A_3 \\
\A_3 & \A_0  & \A_1 & \0 & \0 & \0 & \cdots & \0 & \0 & \A_2 & \0 \\
\0 & \A_3  & \A_0 & \A_1 & \0 & \0 & \cdots & \0 & \A_2 & \0 & \0 \\
\vdots &  & \ddots & \ddots & \ddots &  &  & \bdots & & &\vdots\\
\0 & \cdots  & \0 & \A_3 & \A_0 & \A_1 & \A_2 & \0 & \0 & \cdots & \0 \\
\0 & \cdots  & \0 & \0 & \A_3 & \A_0 + \A_2 & \A_1  & \0 & \0 & \cdots & \0 \\
\0 & \cdots  & \0 & \0 & \A_2 & \A_3 & \A_0 & \A_1 & \0 & \cdots & \0 \\
\vdots &  &  & \bdots &  &  & \ddots & \ddots & \ddots & & \vdots\\
\0 & \0  & \A_2 & \0 & \cdots & \0 & \0 & \A_3 & \A_0 & \A_1 & \0 \\
\0 & \A_2  & \0 & \0 & \cdots & \0 & \0 & \0 & \A_3 & \A_0 & \A_1 \\
-\A_1+\A_2 & \0  & \0 & \0 & \cdots & \0 & \0 & \0 & \0 & \A_3 & \A_0 \\
\end{smallmatrix}
\right)
x_2 \\
=
\left(
\begin{smallmatrix}
\A_0\xb_2 + \A_1(-\xb_2)  + (\A_2 - \A_3)\xb_2 \\
--------------- \\
\A_3\xb_2 + \A_0(-\xb_2)  + \A_1\xb_2 + \A_2(-\xb_2) \\
\A_3(-\xb_2) + \A_0\xb_2  + \A_1(-\xb_2) + \A_2\xb_2 \\
\A_3\xb_2 + \A_0(-\xb_2)  + \A_1\xb_2 + \A_2(-\xb_2) \\
\A_3(-\xb_2) + \A_0\xb_2  + \A_1(-\xb_2) + \A_2\xb_2 \\
\vdots\\
\A_3\xb_2 + \A_0(-\xb_2)  + \A_1\xb_2 + \A_2(-\xb_2) \\
\A_3(-\xb_2) + \A_0\xb_2  + \A_1(-\xb_2) + \A_2\xb_2 \\
--------------- \\
\A_3(-\xb_2) + (\A_0 + \A_2)\xb_2 + \A_1(-\xb_2) \\
--------------- \\
\A_3(-\xb_2) + \A_0\xb_2  + \A_1(-\xb_2) + \A_2\xb_2 \\
\A_3\xb_2 + \A_0(-\xb_2)  + \A_1\xb_2 + \A_2(-\xb_2) \\
\vdots\\
\A_3(-\xb_2) + \A_0\xb_2  + \A_1(-\xb_2) + \A_2\xb_2 \\
\A_3\xb_2 + \A_0(-\xb_2)  + \A_1\xb_2 + \A_2(-\xb_2) \\
--------------- \\
(-\A_1 + \A_2)\xb_2 + \A_3(-\xb_2)  + \A_0\xb_2
\end{smallmatrix}
\right).
\]

Each entry in the previous vector is $ (\A_0 - \A_1 + \A_2 - \A_3)\xb_2 $ which
is zero. This shows that $ B_2x_2 = 0 $. Since the adjacency matrix of $ G $ was
used to establish $ \operatorname{M}(G) = \operatorname{Z}(G) $, by
\hyperref[cor:M_equal_Z_field_independence]{Corollary
\ref{cor:M_equal_Z_field_independence}}  the graph $ G $ has field independent
minimum rank and $ A(G) $ is universally optimal matrix.
\end{proof}

\section{Concluding comments}
Methods were used to determine the equility of the maximum nullity and zero
forcing number of some graph families over an arbitrary field. Although
equalilty of the maximum nullity and zero forcing number of the Extended Cube
graph was determine over the real numbers, equalilty does not hold in general.
However, the following conjecture may be used to characterize the Extended Cube
graphs for which their maximum nullity and zero forcing number are equal over
an arbitrary field.

\begin{conj}
\label{conj:extended_general_bidiakis_cube_field_independent}
Let $ t \equiv 0,1,2 \bmod 6 $ and $ r $ be an integer which is greater than $
\lfloor t/6 \rfloor $. Then the extended cube graphs $
\operatorname{ECG}(t,6r-t-4) $ has field independent minimum rank and their
adjacency matrices are universally optimal.
\end{conj}

Using computational software, $ \operatorname{null}(A(
\operatorname{ECG}(t,6r-t-4) )) = 4 = \operatorname{Z}(
\operatorname{ECG}(t,6r-t-4) ) $ for every $ r \in \{ 2,\dots,12 \}$ when $ t
\in \{ 0,1,2,6,7,8 \}$ .  Note that in
\hyperref[conj:extended_general_bidiakis_cube_field_independent]{Conjecture
\ref{conj:extended_general_bidiakis_cube_field_independent}} when $ t= 6q + 1 $
for some nonnegative integer $ q $ it is the case that $ t \equiv 1 \bmod 6 $.
Moreover,  when $ r  = 2q + 1$, $ 6r-t-4 = 6q+1 $, and $ r > \lfloor t/6
\rfloor = \lfloor \tfrac{6q+1}{6} \rfloor = q  $. In this case, $
\operatorname{ECG}(t,6r-t-4) $ is the same as the graph $
\operatorname{ECG}(6q+1,6q+1) $ as in
\hyperref[thm:generalized_bidiakis_cube_m_z]{Theorem
\ref{thm:generalized_bidiakis_cube_m_z}}.  Also, $ t \equiv 1 \bmod 6 $ implies
$ 6r-t-4 \equiv 1 \bmod 6 $, $ t \equiv 0 \bmod 6 $ implies $ 6r-t-4 \equiv 2
\bmod 6 $, and $ t \equiv 2 \bmod 6 $ implies $ 6r-t-4 \equiv 0 \bmod 6 $.

\bibliographystyle{abbrv}
\bibliography{refs}
\end{document}